\documentclass[a4paper]{article}
\usepackage{amssymb}
\usepackage{amsthm}
\usepackage{amsmath}

\usepackage{mathrsfs}
\usepackage{dsfont}
\usepackage{xcolor}
\usepackage[T1]{fontenc}
\usepackage[utf8]{inputenc}
\usepackage{hyperref}

\newcommand\R{\mathbb{R}}

\newcommand\N{\mathbb{N}}

\newtheorem{thm}{Theorem}[section]
\newtheorem{cor}[thm]{Corollary}
\newtheorem{defi}[thm]{Definition}

\newtheorem{examp}[thm]{Example}
\newtheorem{lem}[thm]{Lemma}
\newtheorem{prop}[thm]{Proposition}
\newtheorem{rk}[thm]{Remark}
\newtheorem{rks}[thm]{Remarks}

\title{Backward Stochastic Evolution Inclusions in UMD Banach Spaces}

\author{E. H.  Essaky \and M. Hassani \and C. E. Rhazlane \and \\
	Universit\'{e} Cadi Ayyad\\ Facult\'{e} Poly-disciplinaire de Safi\\
	Laboratoire de Mod\'{e}lisation et Combinatoire\\
	D\'{e}partement de Math\'{e}matiques et Informatique\\ B.P. 4162,
	Safi, Maroc.\\ \\ e-mails : essaky@uca.ac.ma\qquad m.hassani@uca.ma\qquad charafeddine.rhazlane@gmail.com}
\date{}

\begin{document}

\maketitle

\begin{abstract}
In this paper, we prove the existence of a mild $L^p$-solution for the backward stochastic evolution inclusion (BSEI for short) of the form
\begin{align*}
	\begin{cases}
		dY_t+AY_tdt\in G(t,Y_t,Z_t)dt+Z_tdW_t,\quad t\in [0,T]\\
		Y_T =\xi,
	\end{cases}
\end{align*}
where $W=(W_t)_{t\in [0,T]}$ is a standard Brownian motion, $A$ is the generator of a $C_0$-semigroup on a UMD Banach space $E$,  $\xi$ is a terminal condition from $L^p(\Omega,\mathscr{F}_T;E)$, with $p>1$ and $G$ is a set-valued function satisfying some suitable conditions.

The case when the processes with values in spaces that have martingale type $2$, has been also studied.

\end{abstract}

\section{Introduction}

Backward stochastic differential equations (BSDEs for short)
have been introduced long time ago by J. B. Bismut \cite{bis} both
as the equations for the adjoint process in the stochastic version
of Pontryagin maximum principle as well as the model behind the
Black and Scholes formula for the pricing and hedging of options in
mathematical finance. However, the first published paper on nonlinear
BSDEs appeared in 1990, by Pardoux and Peng \cite{pp1}. A
solution for such an equation is a couple of adapted processes $(Y,
Z)$ with values in $\R\times\R^d$ satisfying
\begin{equation}\label{equa0}
	Y_t = \xi + \displaystyle\int_t^T f(u,Y_u,Z_u)du - \displaystyle\int_t^T Z_u dW_u,
	\quad\quad 0\leq u\leq T.
\end{equation}
In \cite{pp1}, the authors have proved the existence and uniqueness
of the solution under conditions including basically the Lipschitz
continuity of the generator $f$.

Later on, the study of BSDEs has been motivated by their many
applications in mathematical finance, stochastic control and the
second order PDE theory (see, for example, \cite{KPQ, HL, Kob, pp1, PP2} and the references therein).

Extending the results on BSDEs to the theory of Backward Stochastic Differential Inclusions (BSDI for short) has attracted many authors.  In the finite dimension case, we mention the work \cite{kisielewicz} where the following  BSDI has been studied:
\begin{eqnarray}\label{BSDI1}
\begin{cases}
Y_s\in \mathbb{E}\left[ Y_t+\displaystyle\int_s^t G(u,Y_u)du\vert \mathscr{F}_s\right]\\
Y_T\in H(Y_T),
\end{cases}
\end{eqnarray}
where $\mathbb{F}$ is a filtration satisfying the usual hypothesis on $\mathcal{P}_{\mathbb{F}}$, $G$ and $H$ are measurable set-valued mappings from $[0,T]\times \mathbb{R}^m$ into $\mathscr{K}(\mathbb{R}^m)$, where $\mathscr{K}(\mathbb{R}^m)$ denotes the set of all non-empty closed subsets of $\mathbb{R}^m$, and  $\mathbb{E}\left[ Y_t+\int_s^t G(u,Y_u)du\vert \mathscr{F}_s\right]$ denotes the set-valued conditional expectation (see \cite{Hiai} and \cite{papa} for more details) of the set-valued mapping $Y_t(.)+\int_s^t G(u,Y_u(.))du$ with respect to $\mathscr{F}_s$. The author has proved the existence of strong and weak solutions by fulfilling suitable conditions on the data.

We mention also the paper \cite{BAO} where the following BSDI has been considered:
\begin{eqnarray}\label{BSDI2}
\begin{cases}
Y_s \in  \xi+\displaystyle\int_s^1H(t,Y_t,Z_t)dt-\int_s^1Z_tdW_t\\
Y_1 = \xi,
\end{cases}
\end{eqnarray}
where $\xi\in L_{\mathbb{R}^m}^2(\Omega,\mathscr{F},\mathbb{P})$ and $H$ is a borelian set-valued defined on $[0,1]\times\Omega\times \mathbb{R}^m\times \mathbb{R}^{m\times d}$ with convex and compact values in $\mathbb{R}^m$. Using Picard's iterative method, the authors have proved the existence of an adapted solution $(Y^*,Z^*)$ for the BSDI (\ref{BSDI2}) under Lipschitz  condition on the generator $H$ expressed using the Hausdorff distance $\delta_{\mathbb{R}^m}$ on $\mathscr{K}_{cmpt}(\mathbb{R}^m)$ and with the square integrability of $\delta_{\mathbb{R}^m}(H(\,.\,,0_{\mathbb{R}^m},0_{\mathbb{R}^{m\times d}}),\left\lbrace 0_{\mathbb{R}^m}\right\rbrace)$ .\\
\par In the infinite dimensional case, in order to delve into the possibility of generalizing the theory of BSDI in Banach spaces, it was necessary to study and choose adequate stochastic integration theory, see for example \cite{neerven0}, and the work \cite{neerven3} in which the authors indicate that the Banach space $L^p(\Omega,\gamma(0,T;E))$ provided the right setting to establish a rich theory of stochastic integration of adapted processes with values in a UMD space $E$, by giving characterizations of the class of stochastically integrable processes.

Q. L\"u and J. v. Neerven in \cite{BSEE.B} have been invested the fertility of the results in the framework of stochastic integration in UMD Banach Spaces in order to extend both, what was presented by Pardoux and Peng in \cite{pp1} and what was released by Hu and Peng in \cite{BSSE.H}. They studied  the following backward stochastic evolution equation of the form
\begin{eqnarray}\label{BSEE3}
	\begin{cases}
		dY_t+AY_tdt= f(t,Y_t,Z_t)dt+Z_tdW_t,\quad t\in [0,T],\\
		Y_T=\xi,
	\end{cases}
\end{eqnarray}
where  $A$ is the generator of a $C_0$-semigroup on a UMD Banach space $E$ and $W$ is the standard Brownian motion process with values in $\mathbb{R}$. The authors have proved the existence and uniqueness of a specific type of solution for BSEE (\ref{BSEE3}) under suitable integrability and Lipschitz continuity on the coefficient $f$.

In this paper, our goal is to extend the result by L\"u and Nerveen in \cite{BSEE.B}, by proving the existence of a mild $L^p$-solution to the following Backward Stochastic Evolution Inclusions (BSEI for short) of the form
\begin{eqnarray}\label{BSDI3}
		\begin{cases}
		dY_t+AY_tdt\in G(t,Y_t,Z_t)dt+Z_tdW_t,\quad t\in [0,T]\\
		Y_T =\xi,
	\end{cases}
\end{eqnarray}
where $G$ is a set-valued function, and $\xi$ is a terminal condition belongs to $L^p(\Omega,\mathscr{F}_T;E)$ with $p$ is a fixed real number greater than $1$. In order to prove our result we introduce, for every $\delta\in (0,T)$  and each $\mathscr{B}([T-\delta,T])\otimes \mathscr{F}_T$-measurable processes $Y, Z$ defining elements of $L_{\mathbb{F}}^p(\Omega;\gamma(T-\delta;E))$, the subsets $\Lambda_{T-\delta,T}^{Y, Z}$ of the same Banach space $L_{\mathbb{F}}^p(\Omega;\gamma(T-\delta,T;E))$, where its elements can be defined by $\mathscr{B}([T-\delta,T])\otimes \mathscr{F}_T$-measurable mappings which belong to $G(., {Y}, {Z})$ for almost everywhere $(t,\omega)$ in $[T-\delta,T]\times \Omega$.
These subsets allow us to construct, by using Picard's approximation, sequences $\left(Y^n\right)_{n\in \mathbb{N}}$, $\left(Z^n\right)_{n\in \mathbb{N}}$, $\left(g^n\right)_{n\in \mathbb{N}}$ of $E$-valued measurable processes, such that for any natural number $n$ and any real number $t$ in  $[T-\delta,T]$ we have
\begin{eqnarray}\label{Picard}
\left\{
	\begin{array}{lll}
		 (i) &
g^n \text{ defines an element from } \Lambda_{T-\delta, T}^{Y^{n-1},Z^{n-1}} \text{ such that }

	\\ &
	$$
	\left\Vert g^{n}-g^{n-1}\right\Vert_{L_{\mathbb{F}}^p\left(\Omega;\gamma\left(T-\delta, T;E\right)\right)}\leq \rho\left( \Lambda_{T-\delta, T}^{Y^{n-2},Z^{n-2}},\Lambda_{T-\delta, T}^{Y^{n-1},Z^{n-1}}\right)+\varepsilon_{n-1};
	$$
\\ (ii) & 	Y^n\; (resp.\, Z^n)\text{ defines element of the space  }L_{\mathbb{F}}^p\left(\Omega;\gamma\left(T-\delta,T;E\right)\right) \text{ such that}  \\ &
	   Y_t^{n}+\displaystyle\int_t^{T} S(u-t) g_u^n du+\int_t^{T} S(u-t) Z_u^{n}dW_u=S(T-t)\xi\; \textrm{ in }L^p\left(\Omega;E\right),

	\end{array}
	\right.
\end{eqnarray}
where $\rho$ denotes the Hausdorff distance between two subsets of the Banach space $L_{\mathbb{F}}^p\left(\Omega;\gamma\left(T-\delta,T;E\right)\right)$, and $(\varepsilon_n)_{n\in \mathbb{N}}$ is a sequence of strictly positive real numbers. Initially the sequence $(\varepsilon_n)_{n\in \mathbb{N}}$ was chosen arbitrarily, and we then specify it according to a determined criterion which will allow us to have the convergence of the sequences thus constructed.  It should be noted that the second integral that appears instead of the Bochner integral in equation (\ref{Picard})-$(ii)$ is for the sake of readability with a slight abuse of notation, according to Kalton-Weis extension theorem. The next step in proving the existence of a mild $L^p$-solution for BSEI (\ref{BSDI3}) consists in proving that the sequences $\left(Y^n\right)_{n\in \mathbb{N}}$, $\left(Z^n\right)_{n\in \mathbb{N}}$ and $\left(g^n\right)_{n\in \mathbb{N}}$ converge in the space $L_{\mathbb{F}}^p\left(\Omega;\gamma\left({T-\delta},T;E\right)\right)$, for $\delta$ small enough and by an appropriate choice of $\left(\varepsilon_n\right)_{n\in \mathbb{N}}$. This leads us to solve locally the BSEI (\ref{BSDI3}) on $[T-\delta,T]$ and finally to build a mild $L^p$-solution in the interval $[0, T]$ by a concatenation procedure.

The notion of subtrajectory integrals [\cite{set-stochastic-ch2}, section $2.3$], has been included in many works to define the integrals of multivalued functions and give their properties. This led us to wonder about the conditions that allow deriving more properties of the sets $\Lambda_{0,T}^{Y,Z}$ and describing the link between them and the subtrajectory integrals associated with the same set-valued functions $G(.,Y,Z)$.

For the case where $E$ has martingale type $2$, the existence of a  mild $L^p$-solution for Equation  (\ref{BSDI3}) is proved by assuming that the set-valued $G$ satisfy some suitable conditions. We have also studied the special case when E is a Hilbert space with Lipschitz continuity condition using the Hausdorff distance on the set of non-empty compact subsets of $E$:
$$\varrho \left(G\left(.,y,z\right);G\left(.,y',z'\right)\right)\leq \tilde{K}\left(\left\Vert y-y'\right\Vert_{E}+\left\Vert z-z'\right\Vert_{E}\right),$$
and we have discussed the connection with the general case's conditions.

Let us describe our plan. First, most of the material used in this paper is defined in Section $2$. Section $3$ is devoted to prove the existence of a mild $L^p$-solution to the Equation $(\ref{BSDI3})$.

In Section 4, we are concerned by the case where $E$ has martingale type $2$ with more integrated hypotheses, and we solved the BSEI(\ref{BSDI3}) using the continuous embedding $L^2([0,T];E)\hookrightarrow \gamma(0,T;E)$ and the isometry $\gamma(0,T;E)=L^2([0,T];E)$ when $E$ is only a Hilbert space.


\section{Preliminaries}\label{sec:2}
In this section, we will provide some concepts and cognitive tools that will be effective and helpful during the analysis in the sequel. Essentially, we will recall the Pettis integrability, the Bochnner integrability, and some requirements that interfere with the results related to stochastic integration in UMD Banach spaces.

\subsection{Notations}

Throughout this paper, the normed vector spaces are assumed to be real, and we will always identify Hilbert spaces with their dual by means of the Riesz representation theorem.

Let $(E,\left\Vert\,.\,\right\Vert_E)$ be a Banach space, and let $\mathscr{H}$ be a Hilbert space with inner product $<.,.>_{\mathscr{H}}$.
The dual of $E$ will be denoted by $E^*$, the duality between $E$ and $E^*$ will be represented by $<x,x^*>$, and we will assign $0_E$ as an additive identity element of $E$.\\
We shall denote by $vect<x>$ the real vector space spanned by an element $x$ of $E$.\\
$\mathscr{B}(E)$ will denote the Borel $\sigma$-field of $E$ equipped with the norm topology.\\
We shall denote by $\mathcal{L}(\mathscr{H},E)$ (resp. $\mathcal{L}(E)$) the set of all bounded linear operators from $\mathscr{H}$ (resp. $E$) into $E$ (resp. itself).

We shall denote by $ \mathscr{K}_{cmpt}\left( E\right)$\, $\left(\textrm{resp.}\,\mathscr{K}_{ccmpt}\left( E\right)\right)$ the family of all non-empty compact ( resp. non-empty convex and compact) subsets of $E$.

Let $A$ and $B$ be two subsets of the space $\mathscr{K}_{cmpt}\left( E\right)$. The Hausdorff distance $\varrho\left( A,\,B\right)$ between $A$ and $B$ is defined as follows:
\begin{eqnarray}\label{ident100}
	\varrho\left( A,\,B\right) :=\,\max\left(\displaystyle{\sup_{ x\in A}}\;\displaystyle{\inf_{ y\in B}}\left\Vert x-y\right\Vert_E ,\,\displaystyle{\sup_{ y\in B}}\;\displaystyle{\inf_{ x\in A}}\left\Vert x-y\right\Vert_E\right).
\end{eqnarray}
We will use the notation $\left\vert A\right\vert_{\varrho}:=\varrho(A,\left\lbrace 0_E\right\rbrace)=\sup_{x\in A}\left\Vert x\right\Vert_E$.

Let $T$ and $p$ be two fixed real numbers such that $T>0$ and $p> 1$.

For two positive real quantities $L_1$ and $L_2$, we shall write $L_1\lesssim_{p} L_2$ (resp. $L_1\lesssim_{p,E} L_2$) to express that there exists a positive constant $\kappa_p$ (resp. $\kappa_{p,E}$) depending only on $p$ (resp. on $p$ and $E$) such that $L_1\leq \kappa_p L_2$ (resp. $L_1\leq \kappa_{p,E} L_2$ ), while the equivalence $L_1\eqsim_{p} L_2$ (resp. $L_1\eqsim_{p,E} L_2$) will mean that $L_1\lesssim_p L_2$ and $L_2\lesssim_p L_1$ (resp. $L_1\lesssim_{p,E} L_2$ and $L_2\lesssim_{p,E} L_1$).

Let $\left(\Omega,\mathscr{F},\mathbb{P}\right)$ be a probability space, equipped with the filtration $\left(\mathscr{F}_t\right)_{0\leq t\leq T}$ such that $\mathscr{F}_T=\mathscr{F}$.

$\Sigma_{\mathbb{F}}$ will denote the $\sigma$-algebra of all measurable and adapted subsets of $[0,T]\times \Omega$.

For a set-valued function $H:[0,T]\times\Omega\longrightarrow\mathscr{K}_{cmpt}(E)$, the subtrajectory integrals $S^2_{\Sigma_{\mathbb{F}}}(H)$ is defined as follows:
\begin{eqnarray}\label{ident208}
S^2_{\Sigma_{\mathbb{F}}}(H):=\left\lbrace h\in L^2([0,T]\times\Omega,\Sigma_{\mathbb{F}};E):h(u,\omega)\in H(u,\omega)\;a.e.(u,\omega)\in [0,T]\times\Omega\right\rbrace.
\end{eqnarray}

\subsection{Bochner and Pettis integrability}
We will introduce these two notions in a compatible way with the analysis necessary for our work, and we will give the most important results which relate to them. \\

Let $(\mathcal{S},\mathscr{A},\mu)$ be a $\sigma-$finite measure space, and let $q$ be a real number such that $q\geq 1$. 
\begin{defi}
\begin{itemize}
	\item A $\mu$-simple function $f:\mathcal{S}\longrightarrow E$ is of the form $\displaystyle{\sum_{n=1}^N\mathds{1}_{A_n}x_n}$, where $A_n\in \mathscr{A}$, $x_n\in E$ and $\mu(A_n)<+\infty$ for all $1\leq n\leq N$.
	\item A function $f:\mathcal{S}\longrightarrow E$ is measurable ($\mathscr{A}$-measurable) if $f^{-1}(O)\in \mathscr{A}$ for all open set $O$ in $E$.
	\item A real-valued function $f:\mathcal{S}\longrightarrow \R$ is $\mu$-measurable if it is the $\mu$-almost everywhere pointwise limit of real valued $\mu$-simple functions.

	\item[•]A function $f:\mathcal{S}\longrightarrow E$ is weakly measurable (resp. weakly $\mu$-measurable) if the function $s\longmapsto <f(s),x^*>$ is measurable (resp. $\mu$-measurable) for any $x^*\in E^*$.
	\item[•]A function $f:\mathcal{S}\longrightarrow E$ is strongly measurable (resp. strongly $\mu$-measurable) if there
	exists a sequence $(f_n)_{n\geq 1}$ of $\mu$-simple functions converging to $f$ everywhere ($\mu$-almost everywhere).
\end{itemize}
\end{defi}
It should be noted here that $\mu$-measurability, strong $\mu$-measurability and weak $\mu$-measurability (resp. measurability, strong measurability and weak measurability) are equivalent if $E$ is separable.

Let us call two functions which agree $\mu$-almost everywhere $\mu$-versions of
each other. In this case, a function $f:\mathcal{S}\rightarrow E$ is strongly $\mu$-measurable if and only if it has a $\mu$-version which is strongly measurable. For more details we refer to [\cite{neerven1}, Subsection $1.1$].
\begin{defi}
A strongly $\mu$-measurable $f:\mathcal{S}\longrightarrow E$ is called Bochner $q$-integrable if we have 
\begin{eqnarray*}
\left\Vert f\right\Vert^q_{L^q(\mathcal{S};E)}:=\int_{\mathcal{S}}\left\Vert f\right\Vert^q_Ed\mu<\infty,
\end{eqnarray*}
where the integral is in the Lebesgue sense.
\end{defi}
The linear space $L^q(\mathcal{S};E)$ consisting of (classes of a.e. equal) Bochner $q$-integrable functions $f:\mathcal{S}\longrightarrow E$, endowed with the norm $\left\Vert .\right\Vert_{L^q(\mathcal{S};E)}$, is a Banach space.\\

We notice that the $\mu$-simple functions are dense in the Bochner space $L^q(\mathcal{S};E)$, and for any $\mu$-simple function $g=\displaystyle{\sum_{n=1}^N\mathds{1}_{A_n}x_n}$ we define the Bochner integral of $g$ as follows :
$$\int_{\mathcal{S}}^{Bochner}gd\mu=\sum_{j=1}^{N}\mu(A_j)x_{j}.$$
This definition is independent of the representation of $g$.
\\Moreover, the Bochner integral of a function $f\in L^1(\mathcal{S};E)$, is by definition the limit of the Bochner integrals of $f_n$, where $(f_n)_{n\geq 1}$ is a sequence of $\mu$-simple functions convergent in $L^1(\mathcal{S};E)$ to $f$. Besides this, for every $A\in \mathscr{A}$, the Bochner integral of $f$ over $A$ is only the Bochner integral of $\mathds{1}_Af$.\\

\begin{defi}  A function $f:\mathcal{S}\rightarrow E$ is said to be weakly in $L^q(\mathcal{S})$ if it is weakly $\mu$-measurable and $<f,x^*>$ belongs to $L^q(\mathcal{S})$ for every $x^*\in E^*$.
\end{defi}
\begin{defi}\label{def210} Let $f:\mathcal{S}\rightarrow E$ be a function weakly in $L^1(\mathcal{S})$. 
$f$ is called Pettis integrable if for every $A\in \mathscr{A}$, there exists $\nu_f(A)\in E$ such that for every $x^*\in E^*$
$$<\nu_f(A),x^*>=\int_{A}<f,x^*>d\mu.$$
In this case, the set function $\nu_f : \mathscr{A}\rightarrow E$ is called the Pettis integral of $f$ with respect to $\mu$, and
$\nu_f(A)$ is called the Pettis integral of $f$ over $A\in \mathscr{A}$ with respect to $\mu$. We use the notation $\int_A^{Pettis}fd\mu:=\nu_f(A)$.
\end{defi}
\noindent It is worth mentioning that the set of all Pettis integrable functions from $\mathcal{S}$ to $E$ is a linear space, the Pettis integral operator possesses the linear property, and the Bochner integrability implies the Pettis one according to Theorem $1.4.1$ in \cite{set-stochastic-ch1}.\\

Next, let $f: \mathcal{S}\longrightarrow E$ is strongly $\mu$-measurable and
weakly in $L^2(\mathcal{S})$, then for all $g\in L^2(\mathcal{S})$ the function $s \longmapsto g(s)f(s)$ is Pettis
integrable according to Theorem $1.2.37$ in \cite{neerven1}. It thus makes sense to define the linear operator
\begin{eqnarray}\label{map1}
	\begin{array}{ccccc}
		\mathscr{I}_f & : &  L^2(\mathcal{S})& \longrightarrow & E\\
		& & g & \longmapsto & \displaystyle	\mathscr{I}_f(g)= \int_{\mathcal{S}}^{Pettis}g(u)f(u)d\mu(u).
	\end{array}
\end{eqnarray}
The closed graph theorem asserts that $\mathscr{I}_f$ is bounded, and for any $x^*\in E^*$ we have
\begin{eqnarray}\label{ident205}
\left\Vert <f,x^*>\right\Vert_{L^2(\mathcal{S})}\leq \left\Vert \mathscr{I}_f\right\Vert_{\mathcal{L}(L^2(\mathcal{S}),E)}\left\Vert x^*\right\Vert_{E^*}.
\end{eqnarray}
Moreover, $\mathscr{I}_f=\mathscr{I}_{h}$ for any function $h$ equal to $f$ almost everywhere.


\subsection{Functions defining a $\gamma$-radonifying operator}

The $\gamma$-radonifying operators play a crucial role in developing the theory of Gaussian measures in Banach spaces as well as stochastic integrals on some Banach spaces (cf. \cite{kuo, dett, brze1, brze2, neerven0}), and they have been used to characterize some nice geometric structure of the underlying Banach spaces, like type $2$ and cotype $2$ (see \cite{typeco}).\\
In this subsection, we will define the $\gamma$-radonifying operators from the notion of $\gamma$-summing operators. After that, we will introduce some useful notions and results associated with the spaces of $\gamma$-radonifying operators necessary for our study, like function defining a $\gamma$-radonifying operator, which participates in the construction of an integral with a slight abuse of notation.\\

Let $(\gamma_n')_{n\geq 1}$ be a sequence of independent standard Gaussian random variables on a probability space $(\Omega',\mathscr{F}',\mathbb{P}')$. 
\begin{defi}
A linear operator $\mathbb{T} : \mathscr{H}\longrightarrow E$ is called $\gamma$-summing if 
$$\sup \mathbb{E}'\left\Vert \sum_{j=1}^k \gamma'_j \mathbb{T}h_j\right\Vert_E^2 <+\infty,$$
where the supremum is taken over all finite orthonormal systems $\left\lbrace h_1,...,h_k\right\rbrace$ in $\mathscr{H}$.
\end{defi}

Taking into account [\cite{neerven2}, Proposition $9.1.2$], the linear space $\gamma_{\infty}(\mathscr{H},E)$ of all $\gamma$-summing operators from $\mathscr{H}$ to $E$, endowed with the norm 
\begin{eqnarray}\label{gamma1}
\left\Vert \mathbb{T}\right\Vert_{\gamma_{\infty}(\mathscr{H},E)}:= \left(\sup \mathbb{E}'\left\Vert \sum_{j=1}^k \gamma'_j \mathbb{T} h_j\right\Vert_E^2\right)^{\frac{1}{2}},
\end{eqnarray}
is a Banach space, and by considering singletons when taken the supremum in (\ref{gamma1}), one can see that every $\gamma$-summing operator is bounded.
\begin{defi}
An operator in $\mathcal{L}(\mathscr{H}, E)$ is said to be of finite rank if it is a linear combination
of operators of the form $h\otimes e$, where $h\in \mathscr{H}$, $e\in E$, and $h\otimes e$ is defined by $(h\otimes e)(h')=<h,h'>_{\mathscr{H}}e$ for every $h'\in \mathscr{H}$.\\
The space of all finite rank operators contained in $\mathcal{L}(\mathscr{H}, E)$ is denoted by $\mathscr{H}\otimes E$.
\end{defi}

According to Gram-Schmidt orthogonalisation argument, every finite rank operator can be represented in the form $\displaystyle{\sum_{j=1}^kh_j\otimes e_j}$, with $(h_j)_{1\leq j\leq k}$ is orthonormal in $\mathscr{H}$ and $(e_j)_{1\leq j\leq k}$ is a sequence in $E$. Under this representation, we have due to Proposition $9.1.3$ in \cite{neerven2} that
\begin{eqnarray}\label{gamma2}
\left\Vert \sum_{j=1}^kh_j\otimes e_j\right\Vert_{\gamma_{\infty}(\mathscr{H},E)}=\left\Vert \sum_{j=1}^k\gamma_j'e_j\right\Vert_{L^2(\Omega';E)}.
\end{eqnarray}
\begin{defi}
A linear operator $\mathbb{T}:\mathscr{H}\longrightarrow E$ is called $\gamma$-radonifying operator, if there exists a sequence of finite rank operators converging in $\gamma_{\infty}(\mathscr{H},E)$ to $\mathbb{T}$.\\
The space of all $\gamma$-radonifying operators is denoted by $\gamma(\mathscr{H},E)$, and its inherited norm by $\left\Vert .\right\Vert_{\gamma(\mathscr{H},E)}$.
\end{defi}
\noindent To provide some intuition on the role of the sequence $(\gamma_n')_{n\geq 1}$ we refer to [\cite{neerven2}, subsection $9.1.a$]. Here, we mention that the right-hand sides of (\ref{gamma1}) and (\ref{gamma2}) do not depend on the sequence $(\gamma_n')_{n\geq 1}$ and the probability space $(\Omega',\mathscr{F}',\mathbb{P}')$, according to covariance domination theorem [\cite{Abdillah}, Theorem $3.2.2$]. Moreover, the right-hand side of (\ref{gamma2}) is independent of the representation of the finite rank operator as long as we choose $(h_j)_{1\leq j\leq k}$ orthonormal in $\mathscr{H}$ (see [\cite{survey}, page $10$]).\\
We mention also the special case $\mathscr{H}=\R$, where from (\ref{gamma2}) we infer that $\gamma(\R,E)=E$ isometrically.\\

Let us consider the measure space  $([s,t],\mathscr{B}([s,t]),\lambda)$ where $s,t$ are two real numbers such that $0\leq s<t\leq T$, and $\lambda$ is the Lebesgue measure on the $\sigma$-field $\mathscr{B}([s,t])$.\\

The spaces of $\gamma$-radonifying operators which principally involved in our work are when the space $\mathscr{H}$ take the form $L^2([s,t])$, and we will write $\gamma(s,t;E)$ (resp. $\gamma_{\infty}(s,t;E)$) instead of  $\gamma(L^2([s,t]),E)$ (resp. $\gamma_{\infty}(L^2([s,t]),E)$).


\begin{defi}\label{def241}
A function $f:[s,t]\longrightarrow E$ is said to define an element of $\gamma(s,t;E)$ if it satisfies the following conditions:
\begin{itemize}
\item[(i)]
$f$ is strongly $\lambda$-measurable and weakly in $L^2([s,t])$;
\item[(ii)]
the Pettis integral operator $\mathscr{I}_f$ defined as in (\ref{map1}), is $\gamma$-radonifying.
\end{itemize}
In this case we write $f\in \gamma(s,t;E)$ and we define the $\gamma$-norm of $f$ by $\left\Vert f\right\Vert_{\gamma(s,t;E)}:=\left\Vert \mathscr{I}_f\right\Vert_{\gamma(s,t;E)}$.
\end{defi}
\begin{rks}\label{rk241}
\begin{itemize}
\item[$(a)$]We may think of the space $L^2([s,t])\otimes E$, in another way as the space of square integrable functions defined on $[s,t]$ with finite-dimensional range in $E$, and for every $(h,e)\in L^2([s,t])\times E$, the Pettis integral operator of $u\longmapsto h(u)e$ will be equal to $h\otimes e$.\\ As a space of functions, $L^2([s,t])\otimes E$ is dense in $L^2([s,t];E)$ according to [\cite{neerven1}, Lemma $1.2.19$].
\item[$(b)$]Let $f$ (resp. $h$) be a function belonging to $\gamma(s,t;E)$ (resp. $\gamma(0,T;E)$). Due to [\cite{neerven2}, Examples $9.1.11$-$9.1.12$], we have $$
\left\Vert f\right\Vert_{\gamma(s,t;E)}=\left\Vert \mathds{1}_{[s,t]}f\right\Vert_{\gamma(0,T;E)} (resp. \left\Vert h_{\mid [s,t]}\right\Vert_{\gamma(s,t;E)}\leq\left\Vert h\right\Vert_{\gamma(0,T;E)}),$$ 
where $h_{\mid [s,t]}$
is the restriction of $h$ on $[s,t]$.
\end{itemize}
\end{rks}

The computation of the norm of a $\gamma$-radonifying operator can be made in a manner related to this operator or the particularity of $E$. We chose the following one which will be helpful:
\begin{examp}\label{examp1}
Let $A$ be a non-empty $\mathscr{B}([s,t])$-measurable set and $e$ be a vector in $E$.\\
Assume $\lambda(A)>0$ and let's take $h_1=c_1\mathds{1}_{A}$ where $c_1:=1/\sqrt{\lambda(A)}$ is a normalising constant. Thus,
$$\left\Vert \mathds{1}_{A}\otimes e\right\Vert_{\gamma(s,t;E)}=c_1^{-1}\left\Vert h_1\otimes e\right\Vert_{\gamma(s,t;E)}=\sqrt{\lambda(A)}\left(\mathbb{E}'\left(\left\Vert \gamma'_1e\right\Vert_E^2\right)\right)^{\frac{1}{2}}=\sqrt{\lambda(A)}\left\Vert e\right\Vert_E.$$
Assume $\lambda(A)=0$ and let's take $h_2$ of norm $1$ in $L^2([s,t])$. We have
$$(\mathds{1}_A\otimes e)(g)=\int_{[s,t]}^{Pettis}\mathds{1}_{A}(u)g(u)ed\lambda(u)=\left(\int_{[s,t]}\mathds{1}_{A}(u)g(u)du\right)e=0_E,$$
for every $g\in L^2([s,t])$. Thus $\mathds{1}_A\otimes e=h_2\otimes 0_E$. Consequently
$$\left\Vert \mathds{1}_{A}\otimes e\right\Vert_{\gamma(s,t;E)}=\left\Vert h_2\otimes 0_E\right\Vert_{\gamma(s,t;E)}=\left(\mathbb{E}'\left(\left\Vert \gamma'_20_E\right\Vert_E^2\right)\right)^{\frac{1}{2}}=\sqrt{\lambda(A)}\left\Vert e\right\Vert_E.$$
\end{examp}
The notion which introduced in Definition \ref{def241} is extended as follows:
\begin{defi}\label{def242}
A strongly measurable process $\phi:[s,t]\times \Omega\longrightarrow E$ is said to define a random variable $X:\Omega\longrightarrow \gamma(s,t;E)$ if the following conditions hold:
\begin{eqnarray*}
\left\{
	\begin{array}{lll}
		 (i) & \forall x^*\in E^*,\;<\phi,x^*> \in L^2([s,t])\textrm{ for a.s. }\, \omega
;

	  \\ (ii) & \forall h\in L^2([s,t]),\; \forall x^*\in E^*,\,\big \langle  X(\omega)h,x^*\big \rangle=\displaystyle\int_{[s,t]} h(u)<\phi(u,\omega),x^*> d u\quad  \textrm{ for a.s. }\, \omega.

	\end{array}
	\right.
\end{eqnarray*}
\end{defi}

Here, it should be noted that two strongly measurable processes $\phi_1,\phi_2:[s,t]\times\Omega\rightarrow E$, define the same random variable $X:\Omega\rightarrow\gamma(s,t;E)$ if and only if $\phi_1(u,\omega)=\phi_2(u,\omega)$ for almost all $(u,\omega)\in [s,t]\times\Omega$. Conversely, two strongly measurable random variables $X_1,X_2: \Omega\rightarrow\gamma(s,t;E)$ are defined by the same strongly measurable process $\phi$ if and only if $X_1(\omega) = X_2(\omega)$ for $a.s.\,\omega\in \Omega$.

We will write $\phi \in L^p\left(\Omega;\gamma(s,t;E)\right)$ instead of $X\in L^p\left(\Omega;\gamma(s,t;E)\right)$ and the same for its norm, where $\phi$ is a strongly measurable process defining the random variable $X$.\\

Due to Proposition $2.6$ in \cite{neerven3}, the mapping $\mathcal{I}_{\gamma}:L^p(\Omega;\gamma(s,t;E))\rightarrow\mathcal{L} (L^2([s,t]),L^p(\Omega;E))$ defined by
\begin{eqnarray}\label{ident207}
(\mathcal{I}_{\gamma}(X)h)(\omega):=X(\omega)h,\;\omega\in \Omega,\,h\in L^2([s,t]),
\end{eqnarray}
defines the following isomorphism:
\begin{eqnarray}\label{ident210}
L^p(\Omega;\gamma(s,t;E))\eqsim \gamma (s,t;L^p(\Omega;E)).
\end{eqnarray}

\par Now, let us consider the following mapping:
$$m_{s,t}:\phi\in L^2([0,T])\longmapsto \int_{[s,t]} \phi(u)du.$$
By virtue of Kalton-Weis extension theorem [\cite{neerven2}, Theorem $9.6.1$], the mapping $m_{s,t}\otimes I_E$ which associates any operator $\phi\otimes e$ to $(\int_{[s,t]}\phi(u)du)\otimes e$, has a unique extension to a bounded linear operator $\tilde{m}_{s,t}$ from $\gamma(0,T;E)$ to $E$, with the same norm $\sqrt{t-s}$ as for $m_{s,t}$.

For the sake of readability, we will use the notation $\int _s^t f(u)du$, with a slight abuse of notation, instead of the image $\tilde{m}_{s,t}(\mathscr{I}_f)$, for every function $f$ defining element of $\gamma(0,T;E)$.\\

In the following proposition, we will list some properties satisfied by this integral notation.
\begin{prop}\label{prop245}
Let $\mathfrak{B}$ be a bounded operator on $E$. Let $\alpha$ be a real number, $s'$ be a real number with $s<s'<t$, and $f,g$ be two functions defining elements of $\gamma(0,T;E)$.\\
Then the function $\mathfrak{B}f$ defines an element of $\gamma(0,T;E)$ and we have:
\begin{eqnarray}
\int_s^t \mathfrak{B}f(u)du &=& \mathfrak{B} \int_s^t f(u)du, \label{ident212}\\
\left\Vert \int_s^t \mathfrak{B}f(u)du\right\Vert_E &\leq & \sqrt{t-s}\left\Vert \mathfrak{B}\right\Vert_{\mathcal{L}(E)} \left\Vert f\right\Vert_{\gamma(0,T;E)},\label{ident213}\\
\int_s^t (f+\alpha g)(u)du &=& \int_s^t f(u)du+\alpha \int_s^t g(u)du,\label{ident216}\\
\int_s^t f(u)du &=& \int_s^{s'} f(u)du + \int_{s'}^t f(u)du,\label{ident211}
\end{eqnarray}
where integrals of the functions are images of the associated Pettis integral operators by $\tilde{m}_{.,.}$.
\end{prop}

\begin{proof}
Clearly, $\mathfrak{B}f$ is $\lambda$-strongly measurable. Taking into account the fact that
\begin{eqnarray*}
\forall x^*\in E^*,\quad <\mathfrak{B}f,x^*>=<f,\mathfrak{B}^*x^*>,
\end{eqnarray*}
where $\mathfrak{B}^*$ is the adjoint operator of $\mathfrak{B}$, we infer that $\mathfrak{B}f$ is weakly in $L^2([0,T])$.\\
Let $h\in L^2([0,T])$ and $x^*\in E^*$. By Definition \ref{def210}, we derive 
\begin{eqnarray*}
<\int_{[0,T]}^{Pettis} h(u)\mathfrak{B}f(u)d\lambda(u),x^*>&=&\int_{[0,T]} <h(u)f(u),\mathfrak{B}^*x^*>du\\ &=&<\mathfrak{B}\int_{[0,T]}^{Pettis} h(u)f(u)d\lambda(u),x^*>.
\end{eqnarray*}
Thus $\mathscr{I}_{\mathfrak{B}f}=\mathfrak{B}\mathscr{I}_{f}$. Consequently, it suffices to apply Ideal property [\cite{neerven2}, Theorem $9.6.1$] in order to obtain that $\mathfrak{B}f$ defines an element of $\gamma(0,T;E)$ and the estimate (\ref{ident213}) is valid.

On the other hand, due to Kalton-Weis extension theorem we have:
\begin{eqnarray*}
\tilde{m}_{s,t}\left(\mathscr{I}_{\mathfrak{B}f}\right)=\mathscr{I}_{\mathfrak{B}f}\circ (m_{s,t})^*=\mathfrak{B}\mathscr{I}_f\circ (m_{s,t})^*=\mathfrak{B}\tilde{m}_{s,t}\left(\mathscr{I}_f\right),
\end{eqnarray*}
and by the linearity property we have
\begin{eqnarray*}
\tilde{m}_{s,t}\left(\mathscr{I}_{\alpha f+g}\right)=\tilde{m}_{s,t}\left(\alpha \mathscr{I}_{f}+\mathscr{I}_{g}\right)=\alpha\tilde{m}_{s,t}\left( \mathscr{I}_{f}\right)+\tilde{m}_{s,t}\left(\mathscr{I}_{g}\right).
\end{eqnarray*}
Therefore, (\ref{ident212}) and (\ref{ident216}) are proved.\\
Let $(\mathbb{T}_n)_{n\geq 1}$ be a sequence converging in $\gamma(0,T;E)$ to $\mathscr{I}_f$, where for each $n$, $\mathbb{T}_n=\sum_{j=1}^{k_n}h_j^{n}\otimes e_j^n$. Since
\begin{eqnarray}
\tilde{m}_{s,t}(\mathbb{T}_n)=\sum_{j=1}^{k_n}\int_{[s,t]} h_j^{n}(u)du\otimes e_j^n &=&\sum_{j=1}^{k_n}\left(\int_{[s,s']} h_j^{n}(u)du\otimes e_j^n+\int_{[s',t]} h_j^{n}(u)du\otimes e_j^n\right)\nonumber\\
&=&\tilde{m}_{s,s'}(\mathbb{T}_n)+\tilde{m}_{s',t}(\mathbb{T}_n).\label{ident209}
\end{eqnarray}
Then, by tending $n$ to infinity in (\ref{ident209}) we get (\ref{ident211}).
\end{proof}
The reader may wonder under what conditions the integrals of functions in the sense of Proposition \ref{prop245} will be compared with the Bochner integrals of these functions. For this aim, we use the geometric properties cotype $2$ and type $2$ (\cite{neerven2}, Definition $7.1.1$), in order to state the following proposition.
\begin{prop}
Let $f:[0,T]\longrightarrow E$ be a strongly $\lambda$-measurable mapping. The following assertions hold
\begin{itemize}
\item[(a)]If $E$ has cotype $2$ and $f\in \gamma(0,T;E)$, then $f\in L^2([0,T];E)$ and the two integrals are equal;
\item[(b)]If $E$ has type $2$ and $f\in L^2([0,T];E)$, then $f\in \gamma(0,T;E)$ and the two integrals are equal.
\end{itemize}
\end{prop}
\begin{proof}
$(a)$ Let $(f_n)_{n\geq 1}$ be a sequence from $L^2([0,T])\otimes E$ converging in $\gamma(0,T;E)$ to $f$.
By the geometric property cotype $2$, we use the continuously embedding $\gamma(0,T;E)\hookrightarrow L^2([0,T];E)$ (\cite{neerven2}, Theorem $9.2.11$) in order to derive that $f$ is Bochner $2$-integrable. Besides this, we have for every natural number $n$
\begin{eqnarray}
\left\Vert \int_s^tf(u)du-\int_{[s,t]}^{Bochner}f(u)du\right\Vert_E &\leq & \left\Vert \tilde{m}_{s,t}(\mathscr{I}_f)-\tilde{m}_{s,t}(\mathscr{I}_{f_n})\right\Vert_E+\left\Vert \int_{[s,t]}^{Bochner}(f_n-f)(u)du\right\Vert_E\nonumber\\
&\leq & (t-s)^{\frac{1}{2}}\left\Vert f-f_n\right\Vert_{\gamma(0,T;E)}+T^{\frac{1}{2}}\left\Vert f-f_n\right\Vert_{L^2([0,T];E)}\label{ident340}.
\end{eqnarray}
Thus, the integral $\tilde{m}_{s,t}(\mathscr{I}_{f})$ is only the Bochner integral of $f$ over $[s,t]$.\\
$(b)$ Let $(f_n)_{n\geq 1}$ be a sequence from $L^2([0,T])\otimes E$ converging in $L^2([0,T];E)$ to $f$. By virtue of the geometric property type $2$, we derive from [\cite{neerven2}, Theorem $9.2.10$] that $f\in \gamma(0,T;E)$. Moreover, $(f_n)_{n\geq 1}$ converges to $f$ in $\gamma(0,T;E)$. Thus, from (\ref{ident340}) the second aim is obtained.
\end{proof}
\subsection{$\gamma$-boundedness, Upper contraction property and UMD spaces}
Let $\mathscr{R}:[0,T]\longrightarrow \mathcal{L}(E)$ be an operator-valued function. For the simple case where $\mathscr{R}$ is constant, Proposition \ref{prop245} asserts that $\mathscr{R}$ acts as a "pointwise" multiplier from, the set of all functions defining elements of $\gamma(0,T;E)$, to itself. But, the general case requires more data and concepts and, in general, it is not sufficient that the range of $\mathscr{R}$ is uniformly bounded.\\
Let $(\gamma_n')_{n\geq 1}$ and $(\gamma_n'')_{n\geq 1}$ be sequences of independent standard Gaussian random variables on independent probability spaces $(\Omega',\mathscr{F}',\mathbb{P}')$ and $(\Omega'',\mathscr{F}'',\mathbb{P}'')$ respectively.\\
Let $(\gamma_{nm}''')_{n,m\geq 1}$ be a doubly indexed sequence of independent standard Gaussian random variables on a probability space $(\Omega''',\mathscr{F}''',\mathbb{P}''')$.

\begin{defi}\label{def211}
Let $\mathscr{T}$ be a non-empty set of $\mathcal{L}(E)$. We say that $\mathscr{T}$ is $\gamma$-bounded if there exists a finite constant $C\geq 0$ such that for all finite sequences, $(\mathbb{T}_j)_{1\leq j\leq k}$ in $\mathscr{T}$ and $(e_j)_{1\leq j\leq k}$ in $E$, the following inequality holds:
\begin{eqnarray*}
\mathbb{E}'\left(\left\Vert \sum_{j=1}^k \gamma'_j \mathbb{T}_j e_j\right\Vert_E^2\right)\leq C^2\mathbb{E}'\left(\left\Vert \sum_{j=1}^k \gamma'_j e_j\right\Vert_E^2\right).
\end{eqnarray*}
The least admissible constant in the above inequality is called the $\gamma$-bound of $\mathscr{T}$.
\end{defi}
\noindent Obviously, the $\gamma$-boundedness of $\mathscr{T}$ implies the uniform boundedness of $\mathscr{T}$.
\begin{defi}
The space $E$ is said to have the upper contraction property if for all real numbers $q>1$, there exists a finite constant $C_{q,E}\geq 0$ (depending on $q$ and $E$) such that for all finite sequences $(e_{ij})_{1\leq i\leq l\atop 1\leq j\leq k}$ in $E$, we have:
\begin{eqnarray}\label{ident241}
\mathbb{E}'''\left(\left\Vert \sum_{i=1}^l \sum_{j=1}^k \gamma_{ij}''' e_{ij}\right\Vert_E^q\right)\leq C_{q,E}^q\mathbb{E}'\mathbb{E}''\left(\left\Vert \sum_{i=1}^l \sum_{j=1}^k \gamma'_i\gamma''_je_{ij}\right\Vert_E^q\right).
\end{eqnarray}
\end{defi}
\noindent In particular, Hilbert spaces and Lebesgue spaces $L^q(\mathcal{S,\mu})$, with $1\leq q<\infty$, have the contraction property.
\begin{defi}
The space $E$ is said to be a UMD space if for all real numbers $q>1,$ there is a finite constant $\beta_{q,E}\geq 0$ (depending on $q$ and $E$) such that for all $E$-valued $L^{q}$-martingales $(\mathscr{M}_j)_{1\leq j\leq k}$ on a probability space $(\Omega',\mathscr{F}',\mathbb{P}')$, and for all choice of signs $\epsilon_j\in\left\lbrace-1,+1\right\rbrace,\,1\leq j\leq k,$ one has:
\begin{eqnarray}
\mathbb{E}'\left(\left\Vert \sum_{j=1}^k \epsilon_j (\mathscr{M}_j-\mathscr{M}_{j-1})\right\Vert_E^q\right) \leq \beta_{q,E}^q\mathbb{E}'\left(\left\Vert \sum_{j=1}^k (\mathscr{M}_j-\mathscr{M}_{j-1})\right\Vert_E^q\right),
\end{eqnarray}
where $\mathscr{M}_0=0_E$ by convention.
\end{defi}
\noindent In particular, Hilbert spaces, Lebesgue spaces $L^p(\mathcal{S,\mu})$ and Sobolev spaces $\mathscr{W}^{u,p}$, where $u\in \mathbb{R}$, have UMD property. As well, we mention that the UMD space has two most useful characteristics. On the one hand it is reflexive (see Theorem $4.3.3$ in  \cite{neerven1}), and on the other it does not contain a closed subspace isomorphic to the Banach space $c_0$ of real sequences tending to zero, under the supremum norm (see Proposition $7.3.15$ and Theorem $4.6.10$ in \cite{neerven2,neerven1}). The second characteristic allows us to derive that $\gamma_{\infty}(s,t;E)=\gamma(s,t;E)$ isometrically if $E$ is a UMD space (see Theorem $4.3$ in \cite{survey}).

\subsection{Stochastic integration}
Let $E$ be a UMD space and $\left(S(t)\right)_{0\leq t\leq T}$ be a $C_0$-semigroup on $E$ such that $\mathscr{T}:=\left\lbrace S_t:t\in [0,T]\right\rbrace$ is $\gamma$-bounded with $\gamma$-bound $\gamma(S)$.\\
Let $\mathbb{F}:=\Big(\mathscr{F}_t\big)_{0\leq t\leq T}$ be the augmented filtration generated by the Brownian motion $(W_t)_{t\in [0,T]}$.

\begin{defi}\label{def251}
	\begin{enumerate}
		\item[•]
A mapping $\varphi:[0,T]\times \Omega\longrightarrow E$ is called $\mathbb{F}$-adapted step process if it is of the following form
\begin{eqnarray*}
\varphi(t,\omega)=\mathds{1}_{\left\{0\right\}\times A}(t,\omega)\;\tilde{e}+\sum_{i=1}^n\sum_{j=1}^m\mathds{1}_{(t_{i-1},t_i]\times D_{ji}}(t,\omega)\;e_{ji},
\end{eqnarray*}
where $n,m$ are natural numbers; $0\leq t_0<...<t_n\leq T$; the set $A$ is from $\mathscr{F}_0$; the sets $D_{1i},...,D_{mi}$ are disjoints from $\mathscr{F}_{t_{i-1}}$, and the vectors $\tilde{e},e_{ji}$ are in $E$.
       \item[•]
The stochastic integral of the $\mathbb{F}$-adapted step process $\varphi$ of the previous form, with respect to $W$, is defined for each $s\in [0,T]$ as follows
$$\int_0^s \varphi_u dW_u=\sum_{i=1}^n\sum_{j=1}^m\mathds{1}_{D_{ji}}\left(W_{s\wedge t_i}-W_{s\wedge t_{i-1}}\right)\;e_{ji}.$$
       \item[•]
The Banach space $L_{\mathbb{F}}^p\left(\Omega;\gamma(0,T;E)\right)$ is defined as the closure in $L^p\left(\Omega;\gamma(0,T;E)\right)$ of the $E$-valued $\mathbb{F}$-adapted step processes.
	\end{enumerate}
\end{defi}
\noindent We notice that for every $s,t\in [0,T],\,s<t$, the space $L_{\mathbb{F}}^p\left(\Omega;\gamma(s,t;E)\right)$ will be considered in a similar way as the closure in the space $L^p\left(\Omega;\gamma(s,t;E)\right)$ of the restrictions on $(s,t]\times \Omega$ of $E$-valued $\mathbb{F}$-adapted step processes.\\

The following proposition makes it manifest to acquire the legitimacy of some writings like, integrability of some integrands, or some equivalences and inequalities intervening during the stochastic calculation in the next sections. Let $t_0,t_1$ and $t_2$ be real numbers such that $0\leq t_0<t_1<t_2\leq T$.
\begin{prop}\label{prop251}
Let $f:[t_0,t_2]\times \Omega\longrightarrow E$ be a strongly measurable mapping. If $f$ defines an element of $L_{\mathbb{F}}^p\left(\Omega;\gamma(t_0,t_2;E)\right)$ then the $E$-valued mapping
$$
S(.-t_1)f :\,\, (u,\omega)\in  [t_1,t_2]\times \Omega \mapsto S(u-t_1)f(u,\omega),
$$
is strongly measurable and defines an element of $L_{\mathbb{F}}^p\left(\Omega;\gamma(t_1,t_2;E)\right)$.\\ Furthermore, the integral  $\int_{t_1}^{t_2}S(u-t_1)f(u)du$ belongs to $L^p(\Omega;E)$ and $t\mapsto\int_{t}^{t_2} S(u-t)f(u)du$ defines an element of $L^p(\Omega;\gamma(t_0,t_2;E))$ such that
\begin{eqnarray}
\left\Vert \int _{t_1}^{t_2}S(u-t_1)f(u)du\right\Vert_{L^p(\Omega;E)}&\leq & \sqrt{t_2-t_1}\gamma(S)\left\Vert f\right\Vert_{L_{\mathbb{F}}^p(\Omega;\gamma(t_0,t_2;E))},\label{ident217}\\
\left\Vert t\longmapsto\int _{t}^{t_2}S(u-t) f(u)du\right\Vert_{L^p(\Omega;\gamma(t_0,t_2;E))} &\leq & (t_2-t_0)\gamma(S)\left\Vert f\right\Vert_{L_{\mathbb{F}}^p(\Omega;\gamma(t_0,t_2;E))}.\label{ident218}
\end{eqnarray}
\end{prop}
\begin{proof}
Let $e$ be a fixed vector in $E$, and consider for each natural number $m$ the following simple function
\begin{eqnarray*}
h^m=\sum_{j=0}^{2^m-1} \mathds{1}_{\left\lbrace t_1+\frac{j}{2^m}(t_2-t_1)\leq u< t_1+\frac{j+1}{2^m}(t_2-t_1)\right\rbrace} S\Big(\frac{j}{2^m}(t_2-t_1)\Big)e+\mathds{1}_{\left\lbrace t_2\right\rbrace}S(t_2-t_1)e.
\end{eqnarray*}
From the strong continuity of $(S(u-t_1))_{u\in [t_1,t_2]}$, it's easy to see that for any $v\in [t_1,t_2[$:
$$\lim_{m\rightarrow \infty}h^m(v)= S(v-t_1)e.$$
 Thus, the application $u\in [t_1,t_2]\longmapsto S(u-t_1)e$ is strongly measurable. Therefore, $S(.-t_1)f$ is also strongly measurable by taking into account the uniformly boundedness of $(S(t-t_1))_{t\in [t_1,t_2]}$.
 
On the other hand, there exists a negligible set $\mathscr{N}$ from $\mathscr{F}_0$ such that for each $\omega$ from $(\Omega\setminus\mathscr{N})$, the mapping $f(.,\omega)$ defines an element of $\gamma(t_0,t_2;E)$ according to Lemma $2.7$ in  \cite{neerven3}.
\\Let $\omega$ be a fixed element of $(\Omega\setminus\mathscr{N})$.
Once more, according to the strong continuity of $(S(t-t_1))_{t\in [t_1,t_2]}$ we have $S(.-t_1):[t_1,t_2]\longrightarrow\mathcal{L}(E)$ is a strong measurable mapping (in the sense of Definition 8.5.1 in \cite{neerven2}). Thus, by applying Kalton-Weis multiplier theorem [\cite{neerven2}, Theorem $9.5.1$], we infer that
$$ u\in [t_1,t_2]\longmapsto S(u-t_1)f(u,\omega),$$
 belongs to $\gamma_{\infty}(t_1,t_2;E)$ and
  $$\left\Vert S(.-t_1)f(.,\omega)\right\Vert_{\gamma(t_1,t_2;E)}\leq \gamma(S)\left\Vert f(.,\omega)\right\Vert_{\gamma(t_0,t_2;E)}.$$
Let $(g^n)_{n\geq 1}$ be a sequence of $E$-valued $\mathbb{F}$-adapted step processes defined on $[t_0,t_2]\times \Omega$ and converging to $f$ in $L^p(\Omega;\gamma(t_0,t_2;E))$. From Kalton-Weis multiplier theorem and Proposition \ref{prop245}, we have
\begin{eqnarray}
\left\Vert S(.-t_1)f(.,\omega)-S(.-t_1)g^n(.,\omega)\right\Vert_{\gamma(t_1,t_2;E)}&\leq & \gamma(S)\left\Vert f(.,\omega)-g^n(.,\omega)\right\Vert_{\gamma(t_0,t_2;E)},\label{ident214}\\
\left\Vert \int _{t_1}^{t_2}S(u-t_1)\Big((f-g^n)(u,\omega)\Big)du\right\Vert_E &\leq & \sqrt{t_2-t_1}\gamma(S)\left\Vert f(.,\omega)-g^n(.,\omega)\right\Vert_{\gamma(t_0,t_2;E)},\qquad \label{ident215}
\end{eqnarray}
for a.s.$\omega$. We start from (\ref{ident214}) in order to obtain that $S(.-t_1)f$ defines an element of $L^p\left(\Omega;\gamma(t_1,t_2;E)\right)$, and since $S(.-t_1)g^n$ is a $\mathbb{F}$-adapted process then $S(.-t_1)f$ defines an element of $L_{\mathbb{F}}^p\left(\Omega;\gamma(t_1,t_2;E)\right)$ according to [\cite{neerven3}, Propositions $2.11-2.12$]. Moreover, tanks to the inequalities (\ref{ident215}) we infer that $\int _{t_1}^{t_2}S(u-t_1)f(u)du$ is $\mathbb{P}$-strongly measurable and belongs to $L^p(\Omega;E)$ with the estimate (\ref{ident217}).\\
Further, by take a look at Lemma $2.6$ in \cite{BSEE.B}, we derive for a.s.$\omega$ the following inequalities:
\begin{eqnarray*}
\left\Vert t\longmapsto\int _{t}^{t_2}S(u-t)\Big((f-g^n)(u,\omega)\Big)du\right\Vert_{\gamma(t_0,t_2;E)} &\leq & (t_2-t_0)\gamma(S)\left\Vert f(.,\omega)-g^n(.,\omega)\right\Vert_{\gamma(t_0,t_2;E)}.
\end{eqnarray*}
Therefore, $t\longmapsto \int_{t}^{t_2}S(u-t)f(u)du$ belongs to $L^p(\Omega;\gamma(t_0,t_2;E))$ and the estimate (\ref{ident218}) holds.
\end{proof}
\begin{defi}
A strongly measurable and adapted process $\varphi:[0,T]\times \Omega\longrightarrow E$ is said to be $L^p$-stochastically integrable with respect to $W$ if the following conditions are fulfilled
\begin{itemize}
\item[1)]For all $x^*\in E^*$ we have $<\varphi,x^*>\in L^p(\Omega;L^2([0,T]))$.
\item[2)]There exists a sequence of\; $\mathbb{F}$-adapted step processes $\varphi^n:[0,T]\times \Omega\longrightarrow E$ such that:
\begin{itemize}
\item[i)]$(\varphi^n)_{n\geq 1}$ converges in measure on $[0,T]\times\Omega$ to $\varphi$,
\item[ii)]there exists a strongly measurable random variable $\mathbb{I}\in L^p(\Omega;E)$ such that
$$\lim_{n\rightarrow \infty}\int_0^T\varphi_u^ndW_u=\mathbb{I}\textrm{ in }L^p(\Omega;E).$$
\end{itemize}
\end{itemize}
In this case, $\mathbb{I}$ is called the stochastic integral of $\varphi$ with respect to $W$, notation $\int_0^T\varphi_udW_u:=\mathbb{I}$.
\end{defi}
\par Due to Theorem $3.6$ in \cite{neerven3}, a strongly measurable and adapted $E$-valued mapping $\varphi$ is $L^p$-stochastically integrable if and only if it defines an element of $L^p\left(\Omega;\gamma(0,T;E)\right)$. Besides this, if a process is $L^p$-stochastically integrable, its restriction on $[0,t_1]\times \Omega,$ still  $L^p$-stochastically integrable.
\begin{thm}[It\^o isomorphism]\label{thm251}
For any strongly measurable and adapted process $\varphi\in L^p\left(\Omega;\gamma(0,T;E)\right)$ we have the following estimates
$$\mathbb{E}\left(\left\Vert \int_0^T\varphi dW\right\Vert^p\right)\eqsim_p\mathbb{E}\left(\sup_{t\in [0,T]}\left\Vert \int_0^t\varphi dW\right\Vert^p\right)\eqsim_{p,E}\left\Vert \varphi\right\Vert_{L_{\mathbb{F}}^p\left(\Omega;\gamma(0,T;E)\right)}^p.$$
\end{thm}
\begin{rk}\label{rk261}
\begin{itemize}
\item[(a)]The UMD property is necessary in Theorem \ref{thm251} in the sense that it is implied by the validity of the statement in the theorem. For more details see \cite{neerven3}.
\item[(b)]Let $\varphi$ be a strongly measurable and adapted process defining element of $L_{\mathbb{F}}^p\left(\Omega;\gamma(0,t_2;E)\right)$. By virtue of Theorem \ref{thm251} and Kalton-Weis multiplier theorem, we have
\begin{eqnarray*}
\left\Vert \int_{t_1}^{t_2} S(u-t_1)\varphi_u dW_u\right\Vert_{L^p(\Omega;E)} &\eqsim_{p,E}&\left\Vert S(.-t_1)\varphi\right\Vert_{L_{\mathbb{F}}^p\left(\Omega;\gamma(t_1,t_2;E)\right)}\\
&\leq &\gamma(S)\left\Vert \varphi\right\Vert_{L_{\mathbb{F}}^p\left(\Omega;\gamma(0,t_2;E)\right)}.
\end{eqnarray*}
\end{itemize}
\end{rk}

The remaining proposition plays an important role in proving the existence of a solution $(Y, Z)$ of the BSEI (\ref{BSDI3}). It enables to give an explicit formula determining the process $Z$.
\begin{prop}[\cite{BSEE.B}, Lemma $3.5$]\label{prop252}
If $g:[t_1,t_2]\times \Omega\longrightarrow E$ is a strongly measurable mapping defining element in $ L_{\mathbb{F}}^p\left(\Omega;\gamma(t_1,t_2;E)\right)$, then there exists a unique $\tau\in L_{\mathbb{F}}^p\left(\Omega;\gamma(t_1,t_2;\gamma(t_1,t_2;E)\right))$ such that:
\begin{itemize}
\item[(i)]
Almost surely, $\tau$ is supported on the set $\left\lbrace (u,s)\in [t_1,t_2]\times [t_1,t_2]:s\leq u\right\rbrace$.
\item[(ii)]
For almost all $v\in[t_1,t_2]$ we have
$$g(v)=\mathbb{E}(g(v))+\int_{t_1}^v \tau(v,s)dW_s\quad \textrm{ in }L^p(\Omega;E).$$
\item[(iii)]
$\left\Vert \tau\right\Vert_{L^p\left(\Omega;\gamma(t_1,t_2;\gamma(t_1,t_2;E)\right))}\lesssim_{p,E} \left\Vert g\right\Vert_{L^p\left(\Omega;\gamma(t_1,t_2;E)\right)}.$
\end{itemize}
\end{prop}


\section{Backward Stochastic Evolution Inclusions in UMD Spaces}
Let $(\Omega, \mathscr{F}, \mathbb{F}, \mathbb{P})$ be a stochastic basis on which is defined a Brownian motion $(W_t)_{0\leq t\leq T}$ such that $\mathscr{F}=\mathscr{F}_T$ and $\mathbb{F}:=\Big(\mathscr{F}_t\big)_{t\in [0,T]}$ is the augmented natural filtration of $(W_t)_{0\leq t\leq T}$.

\subsection{Measurability of set-valued functions}
Before introducing the problem assigned to this section, we start by defining the measurability of set-valued functions, after that we state a lemma which will be useful to confirm, under suitable conditions, the measurability of some set-valued functions. It should be noted that there are several definitions in various situations, and  the logical relations among the various definitions have been worked out (see  \cite{himme,set-stochastic-ch2} for more details). We  deal with the following definition.
\begin{defi}
Let $F$ be a separable Banach space and $(\mathcal{S},\mathscr{A})$ be a measurable space.\\
Let $H:\mathcal{S}\longrightarrow \mathscr{K}_{cmpt}\left( F\right)$ be a set-valued function.
  \begin{enumerate}
  \item[•]
$H$ is called $\mathscr{A}$-measurable if for any closed set $B$ of $F$ we have:
\begin{eqnarray*}
\left\lbrace s\in \mathcal{S} :\,H(s)\cap B\neq \emptyset\right\rbrace \in \mathscr{A}.
\end{eqnarray*}
  \item[•]
For $\mathcal{S}=[0,T]\times \Omega$, we say that $H$ is adapted if for any closed set $B$ of $F$ we have:
\begin{eqnarray*}
\forall u\in [0,T],\;\left\lbrace \omega\in\Omega :\,H(u,\omega)\cap B\neq \emptyset\right\rbrace \in \mathscr{F}_t.
\end{eqnarray*}
  \end{enumerate}
\end{defi}
\begin{lem}\label{lem401}
Let $\mathscr{P}$ be a $\sigma$-algebra on $[0,T]\times \Omega$ and let $F,\tilde{F}$ be two separable Banach spaces.\\Let $\Pi:[0,T]\times \Omega\times \tilde{F}\rightarrow\mathscr{K}_{cmpt}(F)$ be a $\mathscr{P}\otimes \mathscr{B}(\tilde{F})$-measurable set-valued.
\\If $f:[0,T]\times\Omega\rightarrow \tilde{F}$ is a $\mathscr{P}$-measurable mapping, then the following set-valued
$$\begin{array}{clcl}
  H : &[0,T]\times\Omega  &\longrightarrow  &\mathscr{K}_{cmpt}(F)\\
           &(u,\omega)          &\longmapsto      &\Pi(u,\omega,f(u,\omega))\\
\end{array}$$
is $\mathscr{P}$-measurable.
\end{lem}
\begin{proof}
According to Castaing representation theorem [\cite{intermeas}, Theorem $6.6.8$], there exists a sequence $\left(h_n\right)_{n\geq 1}$ of $\mathscr{P}\otimes\mathscr{B}(\tilde{F})$-measurable selections of $\Pi$ such that:
\begin{eqnarray*}
\Pi\left(u,\omega,\tilde{x}\right)=cl_{F}\left(\bigcup_{m=1}^\infty\left\lbrace h_m\left(u,\omega,\tilde{x}\right)\right\rbrace\right),
\end{eqnarray*}
for any $\left(u,\omega,\tilde{x}\right)\in [0,T]\times \Omega\times \tilde{F}$, where $cl_F$ denotes the closure in $F$.\\
Thus, for any $\left(u,\omega\right)$ we have:
\begin{eqnarray*}
H\left(u,\omega\right)=cl_{F}\left(\bigcup_{m=1}^\infty\left\lbrace h_m\left(u,\omega,f(u,\omega)\right)\right\rbrace\right).
\end{eqnarray*}
Thanks to the properties of measurable functions, we infer that $(u,\omega)\longmapsto h_m(u,\omega,f(u,\omega))$ is $\mathscr{P}$-measurable mapping for each $m\geq 1$. It remains only to apply once more Castaing representation theorem in order to conclude the proof.
\end{proof}
\subsection{Main result}
Let $E$ be a separable UMD space. Let $A$ be the generator of a $C_0$-semigroup $\left\lbrace S_t\right\rbrace_{t\geq 0}$ on $E$.\\

Let us now take up our main topic, the study of the following backward stochastic evolution inclusion
\begin{eqnarray}\label{E1}
BSEI(1):	\begin{cases}
	dY_u+AY_udu\in G(u,Y_u,Z_u)du+Z_udW_u,\quad u\in [0,T]\\
	Y_T=\xi,
\end{cases}
\end{eqnarray}
where $G:\left[0,T\right]\times \Omega\times E\times E\rightarrow \mathscr{K}_{cmpt}\left(E\right)$ is a $\mathscr{B}([0,T])\otimes\mathscr{F}\otimes \mathscr{B}(E)\otimes \mathscr{B}(E)$-measurable set-valued.
\\Following common practice, we suppress the dependence on the underlying probability space from the notation of the set-valued $G$.

For each $\mathscr{B}([0,T])\otimes \mathscr{F}$-measurable processes $Y,Z$ defining elements of $L_{\mathbb{F}}^p\left(\Omega;\gamma\left(0, T;E\right)\right)$, let us consider the following set :
\begin{eqnarray*}
\Lambda^{Y,Z}=\Lambda_{0, T}^{Y,Z}&:=&\left\lbrace g\in L_{\mathbb{F}}^p\left(\Omega;\gamma\left(0, T;E\right)\right):g\in G(.,Y,Z)\,a.e.(u,\omega)\in [0, T]\times \Omega\right\rbrace.
\end{eqnarray*}

For every real number $\delta,\,0<\delta<T,$ and  each $\mathscr{B}([T-\delta,T])\otimes \mathscr{F}$-measurable processes $Y,Z$ defining elements of the space $L_{\mathbb{F}}^p\left(\Omega;\gamma\left(T-\delta, T;E\right)\right)$ we consider the following set
\begin{eqnarray}
	\Lambda_{T-\delta, T}^{Y,Z}&:=&\left\lbrace g\in L_{\mathbb{F}}^p\left(\Omega;\gamma\left(T-\delta, T;E\right)\right):g\in G(.,Y,Z)\,a.e.(u,\omega)\in [T-\delta, T]\times \Omega\right\rbrace.\label{ident301}
\end{eqnarray}
The set $\Lambda_{T-\delta, T}^{Y,Z}$ is considered as a subset of $L_{\mathbb{F}}^p\left(\Omega;\gamma\left(T-\delta, T;E\right)\right)$. It is made up of (equivalence classes of) random variables each of them is an element of $L_{\mathbb{F}}^p\left(\Omega;\gamma\left(T-\delta, T;E\right)\right)$ and can be defined by a such $E$-valued $\mathscr{B}([T-\delta,T])\otimes \mathscr{F}$-measurable process $g$ defined on $[T-\delta, T]\times \Omega$ such that $g(u,\omega)$ belongs to $G(u,Y(u,\omega),Z(u,\omega))$ for almost everywhere $(u,\omega)$ from $[T-\delta,T]\times\Omega$.\bigskip\\
Consider the following assumptions:
\begin{itemize}
\item[$(H_1)$]
$\xi$ belongs to $L^p\left(\Omega,\mathscr{F}_T,\mathbb{P};E\right)$;
\item[$(H_2)$]
The set $\left\lbrace  S_t\right\rbrace_{t\in [0,T]}$ is $\gamma$-bounded i.e. the set of operators $\mathscr{T}:= \{S_t, \,\, t\in [0,T] \}$ is $\gamma$-bounded in the sense of Definition \ref{def211};
\item[$(H_3)$]
The set-valued $G$ has the following properties:
\begin{itemize}
\item[$(i)$]
$\forall\,Y,Z\in L_{\mathbb{F}}^p\left(\Omega;\gamma\left(0,T;E\right)\right),\; \Lambda^{Y,Z}\,\textrm{is a non-empty set};$
\item[$(ii)$]
$\forall\,Y,Z,Y',Z'\in L_{\mathbb{F}}^p\left(\Omega;\gamma\left(0,T;E\right)\right)$,
$$\rho(\Lambda^{Y,Z},\Lambda^{Y',Z'})\leq L\left(\left\Vert Y- Y'\right\Vert_{L_{\mathbb{F}}^p\left(\Omega;\gamma\left(0,T;E\right)\right)}+\left\Vert Z- Z'\right\Vert_{L_{\mathbb{F}}^p\left(\Omega;\gamma\left(0,T;E\right)\right)}\right);$$
\end{itemize}
where $\rho(\Lambda^{Y,Z},\Lambda^{Y',Z'})$ denotes the Hausdorff distance of $\Lambda^{Y,Z}$ and $\Lambda^{Y',Z'}$, and is defined similarly to (\ref{ident100}) taking into account the different spaces.
\end{itemize}
We put at our disposal another hypothesis:
\begin{itemize}
\item[$(H_4)$] The set-valued $G$ has the following property: $\forall\,Y,Z\in L_{\mathbb{F}}^p\left(\Omega;\gamma\left(0,T;E\right)\right),\; \Lambda^{Y,Z}\,\textrm{is a closed set}.$
\end{itemize}
\begin{lem}\label{lem301}
The following assertions hold:
\begin{itemize}
\item[1)]	If $(H_3)$ is satisfied then the set-valued $G$ has also the following properties:
	\begin{itemize}
		\item[$(a)$]
		$\forall s,t\in [0,T]$, $s<t$, $\forall\,Y,Z\in L_{\mathbb{F}}^p\left(\Omega;\gamma\left(s,t;E\right)\right),$
		\begin{eqnarray*}
		\Lambda_{s,t}^{Y,Z}&:=&\left\lbrace g\in L_{\mathbb{F}}^p\left(\Omega;\gamma\left(s, t;E\right)\right):g  \in G(.,Y,Z)\,a.e.(u,\omega)\in [s, t]\times \Omega\right\rbrace\neq \emptyset;
		\end{eqnarray*}
		\item[$(b)$]
		$\forall s,t\in [0,T]$, $s<t$, $\forall\,Y,Z,Y',Z'\in L_{\mathbb{F}}^p\left(\Omega;\gamma\left(s,t;E\right)\right)$,
		$$\rho(\Lambda_{s,t}^{Y,Z},\Lambda_{s,t}^{Y',Z'})\leq L\left(\left\Vert Y- Y'\right\Vert_{L_{\mathbb{F}}^p\left(\Omega;\gamma\left(s,t;E\right)\right)}+\left\Vert Z- Z'\right\Vert_{L_{\mathbb{F}}^p\left(\Omega;\gamma\left(s,t;E\right)\right)}\right).$$
	\end{itemize}
\item[2)]  If the hypothesis $(H_4)$ is fulfilled then
\begin{eqnarray*}
\forall s,t\in [0,T],\, s<t,\, \forall\,Y,Z\in L_{\mathbb{F}}^p\left(\Omega;\gamma\left(s,t;E\right)\right),\;\Lambda_{s,t}^{Y,Z}\textrm{ is closed too.}
\end{eqnarray*}
\end{itemize}
\end{lem}
\begin{proof}
Let $s,t$ be two real numbers from $[0,T]$ such that $s<t$.
\\Let $\mathcal{Y},\mathcal{Z}$ be two $\mathscr{B}([0,T])\otimes \mathscr{F}$-measurable processes defining elements of $L_{\mathbb{F}}^p\left(\Omega;\gamma\left(0,T;E\right)\right)$.\\
Let $h$ be a $\mathscr{B}([0,T])\otimes \mathscr{F}$-measurable process defining an element of $\Lambda^{\mathcal{Y},\mathcal{Z}}$.
\\Taking into account Remarks \ref{rk241}-$(b)$, we derive for a.s. $\omega$ that $$\left\Vert h(.,\omega)_{\mid [s,t]}\right\Vert_{\gamma\left(s,t;E\right)}\leq \left\Vert h(.,\omega)\right\Vert_{\gamma\left(0,T;E\right)}$$
then the restriction $h_{\mid s,t}$ of $h$ on $[s,t]\times \Omega$ defines an element of $L_{\mathbb{F}}^p\left(\Omega;\gamma\left(s,t;E\right)\right)$. Besides this $h_{\mid s,t}$ belongs to $G(.,\mathcal{Y},\mathcal{Z})$ for a.e. $(u,\omega)$ from $[s,t]\times \Omega$, therefore $h_{\mid s,t}$ defines an element of $\Lambda_{s,t}^{\mathcal{Y}_{\mid s,t},\mathcal{Z}_{\mid s,t}}$.\bigskip\\
\noindent $1)$ 
Let $Y,Z$ be two $\mathscr{B}([s,t])\otimes \mathscr{F}$-measurable processes defining elements of $L_{\mathbb{F}}^p\left(\Omega;\gamma\left(s,t;E\right)\right)$. Let $(Y^n)_{n\geq 1}$ be a sequence of adapted step processes converging to $Y$ with respect to the strong topology of $L_{\mathbb{F}}^p\left(\Omega;\gamma\left(s,t;E\right)\right)$. Define the $E$-valued mappings $\tilde{Y},\;\tilde{Z}$ and the $\mathbb{F}$-adapted step processes $\tilde{Y}^n$ on $[0,T]\times \Omega$ as follows:
\begin{eqnarray*}
\tilde{Y}= \mathds{1}_{[s,t]}Y,\quad \tilde{Z}=\mathds{1}_{[s,t]} Z,\quad \tilde{Y}^n=\mathds{1}_{[s,t]}Y^n.
\end{eqnarray*} 
Taking into account Remarks \ref{rk241}-$(b)$, we derive for every natural number $n$ that:
\begin{eqnarray}
\int_{\Omega}\left\Vert Y(.,\omega)\right\Vert_{\gamma\left(s,t;E\right)}^p \mathbb{P}(d\omega)&=&\int_{\Omega}\left\Vert \tilde{Y}(.,\omega)\right\Vert_{\gamma\left(0,T;E\right)}^p \mathbb{P}(d\omega)\label{ident313}\\
\int_{\Omega}\left\Vert Y(.,\omega)-Y^n(.,\omega)\right\Vert_{\gamma\left(s,t;E\right)}^p \mathbb{P}(d\omega)&=&\int_{\Omega}\left\Vert \tilde{Y}(.,\omega)-\tilde{Y}^n(.,\omega)\right\Vert_{\gamma\left(0,T;E\right)}^p \mathbb{P}(d\omega).\label{ident314}
\end{eqnarray}
Next, from (\ref{ident313}) we infer that the process $\tilde{Y}$ defines an element of $L^p\left(\Omega;\gamma\left(0,T;E\right)\right)$, and due to (\ref{ident314}) we derive that $(\tilde{Y}^n)_{n\geq 1}$ converges to $\tilde{Y}$ in $L^p\left(\Omega;\gamma\left(0,T;E\right)\right)$. Thus $\tilde{Y}$ defines an element of $L_{\mathbb{F}}^p\left(\Omega;\gamma\left(0,T;E\right)\right)$, and $\tilde{Z}$ so does similarly.\\
Since $\Lambda^{\tilde{Y},\tilde{Z}}$ is non-empty, then $\Lambda_{s,t}^{\tilde{Y}_{\mid s,t},\tilde{Z}_{\mid s,t}}$ is non-empty too, and $\Lambda_{s,t}^{Y,Z}$ so is.\\
Let $Y',Z'$ be two $\mathscr{B}([s,t])\otimes \mathscr{F}$- measurable processes defining elements of $L_{\mathbb{F}}^p\left(\Omega;\gamma\left(s,t;E\right)\right)$, let $g$ be a $\mathscr{B}([s,t])\otimes \mathscr{F}$-measurable process defining element of $\Lambda_{s,t}^{Y,Z}$, and define the following processes:
\begin{eqnarray*}
\tilde{Y'}=\mathds{1}_{[s,t]}Y',\quad \tilde{Z'}=\mathds{1}_{[s,t]}Z',\quad \tilde{g}=\mathds{1}_{[s,t]}g.
\end{eqnarray*}
We have:
\begin{eqnarray*}
d_{L_{\mathbb{F}}^p\left(\Omega;\gamma\left(0,T;E\right)\right)}(\tilde{g},\Lambda^{\tilde{Y'},\tilde{Z'}})\geq \inf_{h\in \Lambda^{\tilde{Y'},\tilde{Z'}} }\left\Vert (\tilde{g}-h)_{\mid s,t} \right\Vert_{L_{\mathbb{F}}^p\left(\Omega;\gamma\left(s,t;E\right)\right)},
\end{eqnarray*}
where $d_{L_{\mathbb{F}}^p\left(\Omega;\gamma\left(0,T;E\right)\right)}(\tilde{g},\Lambda^{\tilde{Y'},\tilde{Z'}})$ denotes the distance between the element of $L_{\mathbb{F}}^p\left(\Omega;\gamma\left(0,T;E\right)\right)$ defined by $\tilde{g}$ and the subset $\Lambda^{\tilde{Y'},\tilde{Z'}}$. Thus,
\begin{eqnarray*}
d_{L_{\mathbb{F}}^p\left(\Omega;\gamma\left(s,t;E\right)\right)}(g,\Lambda_{s,t}^{Y',Z'})&\leq & L\left(\left\Vert \tilde{Y}- \tilde{Y'}\right\Vert_{L_{\mathbb{F}}^p\left(\Omega;\gamma\left(0,T;E\right)\right)}+\left\Vert \tilde{Z}- \tilde{Z'}\right\Vert_{L_{\mathbb{F}}^p\left(\Omega;\gamma\left(0,T;E\right)\right)}\right)\\
&=& L\left(\left\Vert Y- Y'\right\Vert_{L_{\mathbb{F}}^p\left(\Omega;\gamma\left(s,t;E\right)\right)}+\left\Vert Z- Z'\right\Vert_{L_{\mathbb{F}}^p\left(\Omega;\gamma\left(s,t;E\right)\right)}\right).
\end{eqnarray*}
Therefore the properties $(a)$ and $(b)$ are proved.\\
\noindent $2)$ To prove the second claim, let $(g^n)_{n\geq 1}$ be a Cauchy sequence in $\Lambda_{s,t}^{Y,Z}$ and let $\tilde{h}$ be in $\Lambda_{0,T}^{0_E,0_E}$.
\\Next, let's take:
$$\tilde{Y}=\mathds{1}_{[s,t]}Y,\quad \tilde{Z}=\mathds{1}_{[s,t]}Z,\quad \tilde{g}^n=\mathds{1}_{[s,t]}g^n+\mathds{1}_{([0,T]\setminus[s,t])}\tilde{h}.$$
Since $(\tilde{g}^n)_{n\geq 1}$ is a Cauchy sequence in the closed set $\Lambda_{0,T}^{\tilde{Y},\tilde{Z}}$, therefore it converges to $\tilde{g}\in \Lambda_{0,T}^{\tilde{Y},\tilde{Z}}$. Consequently, $(g^n)_{n\geq 1}$ converges to the restriction of $\tilde{g}$ on $[s,t]\times \Omega$ which belongs to $\Lambda_{s,t}^{Y,Z}$.
\end{proof}
\begin{rk}
Assume that $(H_3)$ and $(H_4)$ are fulfilled. Let $s,t$ be two real numbers such that $0\leq s<t\leq T$ and let $Y,Z\in L_{\mathbb{F}}^p(\Omega;\gamma(s,t;E))$. Let $(Y^n)_{n\geq 1}$ and $(Z^n)_{n\geq 1}$ be two sequences of adapted step processes converging respectively to $Y$ and $Z$ in $L_{\mathbb{F}}^p(\Omega;\gamma(s,t;E))$.\\
By virtue of $(H_4)$ we derive the closedness of the sets $\Lambda^{Y,Z}_{s,t}$, $\Lambda^{Y^n,Z^n}_{s,t}$ in $L_{\mathbb{F}}^p(\Omega;\gamma(s,t;E))$, and by $(H_3)$ the sequence $(\rho(\Lambda^{Y,Z}_{s,t},\Lambda^{Y^n,Z^n}_{s,t}))_{n\geq 1}$ converges to $0$. Then, due to Proposition $1.1.3$ in \cite{papa} we have:
\begin{eqnarray*}
\Lambda_{s,t}^{Y,Z}=\bigcap_{n\geq 1}\textrm{\large{cl}}_{L_{\mathbb{F}}^p(\Omega;\gamma(s,t;E))}(\bigcup_{m\geq n}\Lambda_{s,t}^{Y^m,Z^m}),
\end{eqnarray*}
where $\textrm{\large{cl}}_{L_{\mathbb{F}}^p(\Omega;\gamma(s,t;E))}$ denotes the closure in $L_{\mathbb{F}}^p(\Omega;\gamma(s,t;E))$.
\end{rk}

 Let us now introduce the definition of our BSEI.
\begin{defi}\label{def301}
A pair $\left(Y,Z\right)$ of $\mathscr{B}([0,T])\otimes \mathscr{F}$-measurable and $\mathbb{F}$-adapted processes is called a mild $L^p$-solution for the inclusion (\ref{E1}) if it satisfies the following properties:
\begin{itemize}
\item[(a)]
$Y,Z\in L_{\mathbb{F}}^p\left(\Omega;\gamma\left(0,T;E\right)\right)$;
\item[(b)]
$Y$ belongs to $C([0,T];L^p\left(\Omega;E\right))$;
\item[(c)]
There exists a $\mathscr{B}([0,T])\otimes \mathscr{F}$-measurable process $g:[0,T]\times \Omega\longrightarrow E$ such that $g\in \Lambda^{Y,Z}$ and for any $t\in [0,T]$:
\begin{eqnarray}\label{ident303}
Y_t+\int_t^T S(u-t) g_udu+\int_t^T S(u-t) Z(u)dW_u= S(T-t)\xi \;\textrm{ in }\;L^p\left(\Omega;E\right).
\end{eqnarray}
\end{itemize}
\end{defi}

\noindent It should be noted that the integral as a second term in the left-hand side of the equation (\ref{ident303}) is in the sense of Proposition \ref{prop245}.\\
\bigskip

The first main theorem of this section is the following.
\begin{thm}\label{thm301}
Let $(H_1)-(H_4)$ be satisfied and assume in addition that $E$ has the upper contraction property. Then, the backward stochastic evolution inclusion (\ref{E1}) admits a mild $L^p$-solution $(Y,Z)$.
\end{thm}
\begin{proof}
We shall use the successive approximation method, called the Picard iteration method, to construct a mild $L^p$-solution to the BSEI (\ref{E1}). For this purpose, we will divide the proof into four steps.


\noindent {\textbf{First step:}} In the beginning, we will show that the sets defined by (\ref{ident301}) intervene by iteratively applying the characterization of the lower bound property in order to obtain sequences allowing to approximate the different factors involved in the resolution of the BSEI (\ref{E1}). These sets are non-empty whatever the parameters that make up its determination according to Lemma \ref{lem301}.
\par Let $0<\delta \leq T$ (which will be chosen later) and  $(\varepsilon_n)_{n\in \mathbb{N}}$ be a sequence of strictly positive real numbers.\\
We initialize by taking the three processes defined on $[T-\delta, T]\times\Omega$ as follows:
$$Y^0=g^0=Z^0= 0_E.$$
It yields that we can construct by induction sequences $\left(Y^n\right)_{n\in \N },\;\left(Z^n\right)_{n\in \N}$ of $E$-valued measurable and adapted processes, and $\left(g^n\right)_{n\in \N}$ of $E$-valued measurable processes such that for any natural number $n$ and for any real number $t$ from $[T-\delta, T]$ we have
\begin{eqnarray}\label{eq1}
\left\{
	\begin{array}{lll}
		 (i) &
g^n \text{ defines an element from } \Lambda_{T-\delta, T}^{Y^{n-1},Z^{n-1}} \text{ such that }

	\\ &
	$$
	\left\Vert g^{n}-g^{n-1}\right\Vert_{L_{\mathbb{F}}^p\left(\Omega;\gamma\left(T-\delta, T;E\right)\right)}\leq \rho\left( \Lambda_{T-\delta, T}^{Y^{n-1},Z^{n-1}},\Lambda_{T-\delta, T}^{Y^{n-2},Z^{n-2}}\right)+\varepsilon_{n-1} ;
	$$
\\ (ii) & 	Y^n\; (resp.\, Z^n)\text{ defines element of the space  }L_{\mathbb{F}}^p\left(\Omega;\gamma\left(T-\delta,T;E\right)\right) \text{ such that}  \\ &
	   Y_t^{n}+\displaystyle\int_t^{T} S(u-t) g_u^n du+\int_t^{T} S(u-t) Z_u^{n}dW_u=S(T-t)\xi\; \textrm{ in }L^p\left(\Omega;E\right).

	\end{array}
	\right.
\end{eqnarray}

By taking a look at Proposition \ref{prop251} one can see obviously that the second and third terms on the left-hand side of each equation in $(ii)$ are well-defined as elements of $L^p(\Omega;E)$.

Now, we will show that the sequences of the processes previously announced are sufficient to lead us to at least one mild $L^p$-solution of the stochastic evolution inclusion studied. The proof will be divided into four steps.


\noindent {\textbf{Second step:}}
The claim in this step is that for $\delta$ small enough,  the sequences $\left(Y^n\right)_{n\in \N}$, $\left(Z^n\right)_{n\in \N}$ converge in $L_{\mathbb{F}}^p\left(\Omega;\gamma\left( T-\delta, T;E\right)\right)$ to limits $Y$ and $Z$ respectively.\\
For every $n$, we start from the following formulas :
\begin{eqnarray}
Y_t^{n+1} +\int_t^{T} S(u-t)g_u^{n+1} du +\int_t^{T} S(u-t)Z_u^{n+1}dW_u &=& S(T-t)\xi\label{ident316}\\
Y_t^n+\int_t^{T} S(u-t)g_u^ndu +\int_t^{T} S(u-t)Z_u^{n}dW_u &=& S(T-t)\xi\label{ident317}\\
Y_{T}^{n+1}=Y_{T}^n &=&\xi.\nonumber
\end{eqnarray}
Subtracting (\ref{ident316}) and (\ref{ident317}) we get
\begin{eqnarray}
Y_t^{n+1}-Y_t^n +\int_t^{T} S(u-t)\left(g_u^{n+1}-g_u^n\right)du +\int_t^{T} S(u-t)\left(Z_u^{n+1}-Z_u^{n}\right)dW_u = 0_{L^p(\Omega;E)}. \label{ident318}
\end{eqnarray}
According to Proposition \ref{prop252}, we infer that there exists $\tau^{\mathtt{i}}\in L_{\mathbb{F}}^p\left(\Omega;\gamma\left({T-\delta},T;\gamma\left({T-\delta},T;E\right)\right)\right)$ for each $\mathtt{i}=n,n+1$, such that for almost all $u\in [{T-\delta},T]$ we have in $L^p(\Omega;E)$ the following equalities
\begin{eqnarray*}
g_u^\mathtt{i} &=&\mathbb{E}\left(g_u^\mathtt{i}\right)+\int_{T-\delta}^u \tau^\mathtt{i}(u,s)dW_s,\,\mathtt{i}=n,n+1.
\end{eqnarray*}
It follows, for almost all $u\in [{T-\delta},T]$ we have:
\begin{eqnarray*}
g_u^{n+1}-g_u^n=\mathbb{E}\left(g_u^{n+1}-g_u^n\right)+\int_{T-\delta}^u \left(\tau^{n+1}(u,s)-\tau^n(u,s)\right)dW_s,
\end{eqnarray*}
in $L^p(\Omega;E)$. Thus, once more due to Proposition \ref{prop252} we derive that
\begin{eqnarray}\label{ident322}
\left\Vert \tau^{n+1}-\tau^n\right\Vert_{L^p\left(\Omega;\gamma({T-\delta},T;\gamma({T-\delta},T;E)\right))}\lesssim_{p,E} \left\Vert g^{n+1}-g^n\right\Vert_{L^p\left(\Omega;\gamma({T-\delta},T;E)\right)}.
\end{eqnarray}
On the other hand, by uniqueness of the mild $L^p$-solution $\left(Y^{n+1}-Y^n ,Z^{n+1}-Z^{n}\right)$ one can see that
\begin{eqnarray*}
Z_u^{n+1}-Z_u^{n}=\int_u^{T} S(s-u)\left(\tau^{n+1}(s,u)-\tau^n(s,u)\right)ds.
\end{eqnarray*}
Thus, according to \cite{BSEE.B} (Lemma $2.6$, Lemma $2.9$) and thanks to inequality (\ref{ident322}) we have successively
\begin{eqnarray}\label{estimates}
\left\Vert Z^{n+1}-Z^n\right\Vert_{L_{\mathbb{F}}^p\left(\Omega;\gamma\left({T-\delta},T;E\right)\right)} &\leq & \delta^{\frac{1}{2}}\gamma(S)\left\Vert \tau^{n+1}- \tau^n\right\Vert_{L^p\left(\Omega;\gamma\left(\Gamma_{\delta};E\right)\right)}\nonumber\\
&\lesssim_{p,E}& \delta^{\frac{1}{2}}\gamma(S)\left\Vert \tau^{n+1}- \tau^n\right\Vert_{L_{\mathbb{F}}^p\left(\Omega;\gamma\left({T-\delta},T;\gamma\left({T-\delta},T;E\right)\right)\right)}\nonumber\\
&\lesssim_{p,E} & \delta^{\frac{1}{2}}\gamma(S)\left\Vert g^{n+1}-g^n\right\Vert_{L_{\mathbb{F}}^p\left(\Omega;\gamma\left({T-\delta},T;E\right)\right)},
\end{eqnarray}
where $\gamma(S)$ is the $\gamma$-bound of $\mathscr{T}$ and $\Gamma_{\delta}:=\left\lbrace (u,v)\in [{T-\delta},T]\times [{T-\delta},T]: {T-\delta}<v\leq u<T \right\rbrace$.\\Consequently, by using (\ref{eq1}) we have
\begin{eqnarray}
& & \left\Vert Z^{n+1}-Z^n\right\Vert_{L_{\mathbb{F}}^p\left(\Omega;\gamma\left({T-\delta},T;E\right)\right)}\label{ident324}\\
&\lesssim_{p,E} & \delta^{\frac{1}{2}}\gamma(S)\left(L \left\Vert Y^n -Y^{n-1}\right\Vert_{L_{\mathbb{F}}^p\left(\Omega;\gamma\left({T-\delta},T;E\right)\right)}+L \left\Vert Z^n- Z^{n-1}\right\Vert_{L_{\mathbb{F}}^p\left(\Omega;\gamma\left({T-\delta},T;E\right)\right)}+\varepsilon_n\right).\nonumber
\end{eqnarray}
Due to Proposition \ref{prop245} and the isomorphism (\ref{ident210}) we infer that
\begin{eqnarray}
& &\left\Vert t\longmapsto \int_t^{T} S(u-t)\left( Z_u^{n+1}-Z_u^{n}\right)dW_u\right\Vert_{L^p\left(\Omega;\gamma\left({T-\delta},T;E\right)\right)}\label{ident325}\\
&=&\left\Vert t\longmapsto \int_t^{T} S(u-t)\left( \int_u^{T} S(s-u)\left(\tau^{n+1}(s,u)-\tau^n(s,u)\right)ds\right)dW_u\right\Vert_{L^p\left(\Omega;\gamma\left({T-\delta},T;E\right)\right)}\nonumber\\
&=&\left\Vert t\longmapsto \int_t^{T} \int_u^{T} S(s-t)\left(\tau^{n+1}(s,u)-\tau^n(s,u)\right)ds dW_u\right\Vert_{L^p\left(\Omega;\gamma\left({T-\delta},T;E\right)\right)}\nonumber\\
&\eqsim_p&\left\Vert t\longmapsto \int_t^{T} \int_u^{T} S(s-t)\left(\tau^{n+1}(s,u)-\tau^n(s,u)\right)ds dW_u\right\Vert_{\gamma({T-\delta},T;L^p(\Omega;E))}=:(M_2).\nonumber
\end{eqnarray}
Taking into account Theorem \ref{thm251}, Proposition $\ref{prop245}$ and the isomorphism (\ref{ident210}) we derive that
 \begin{eqnarray}
(M_2)&\eqsim_{p,E}&\left\Vert t\longmapsto \left[u\longmapsto \mathds{1}_{t\leq u}\int_u^{T} S(s-t)\Delta{\tau}^{n}(s,u)ds \right]\right\Vert_{\gamma({T-\delta},T;L^p\left(\Omega;\gamma\left({T-\delta},T;E\right)\right))}\label{ident326}\\
&=&\left\Vert t\longmapsto \left[u\longmapsto \mathds{1}_{t\leq u}\;S(u-t)\int_u^{T} S(s-u)\Delta{\tau}^{n}(s,u)ds \right]\right\Vert_{\gamma({T-\delta},T;L^p\left(\Omega;\gamma\left({T-\delta},T;E\right)\right))}\nonumber\\
&\eqsim_p&\left\Vert t\longmapsto \left[u\longmapsto \mathds{1}_{t\leq u}\;S(u-t)\int_u^{T} S(s-u)\Delta{\tau}^{n}(s,u)ds \right]\right\Vert_{L^p\left(\Omega;\gamma({T-\delta},T;\gamma\left({T-\delta},T;E\right)\right))}\nonumber\\
&=:&(M_3)\nonumber
\end{eqnarray}
where  $\Delta{\tau}^{n}:=\tau^{n+1}-\tau^{n}$.
\\According to both, Kalton-Weis multiplier theorem and Example $(\ref{examp1})$, one can see that:
\begin{eqnarray}
(M_3)&\leq & \gamma(S)\left\Vert t\longmapsto \left[u\longmapsto \int_u^{T} S(s-u)\left(\tau^{n+1}(s,u)-\tau^n(s,u)\right)ds \right]\right\Vert_{L^p\left(\Omega;\gamma({T-\delta},T;\gamma\left({T-\delta},T;E\right)\right))}\label{ident327}\\
&=& \delta^{\frac{1}{2}}\gamma(S)\left\Vert Z^{n+1}-Z^n \right\Vert_{L_{\mathbb{F}}^p\left(\Omega;\gamma\left({T-\delta},T;E\right)\right)}.\nonumber
\end{eqnarray}
 It follows, by combining (\ref{ident325}), (\ref{ident326}) and (\ref{ident327})  with (\ref{ident324}), that
\begin{eqnarray}
& &\left\Vert t\longmapsto \int_t^{T} S(u-t)\left( Z_u^{n+1}-Z_u^{n}\right)dW_u\right\Vert_{L^p\left(\Omega;\gamma\left({T-\delta},T;E\right)\right)}\label{ident328}\\
&\lesssim_{p,E} &  \delta\gamma(S)^2\left(L\left\Vert Y^n -Y^{n-1}\right\Vert_{L_{\mathbb{F}}^p\left(\Omega;\gamma\left({T-\delta},T;E\right)\right)}+L\left\Vert Z^n- Z^{n-1}\right\Vert_{L_{\mathbb{F}}^p\left(\Omega;\gamma\left({T-\delta},T;E\right)\right)}+\varepsilon_n\right).\nonumber
\end{eqnarray}
On the other hand, it turns out from Proposition \ref{prop251} that:
\begin{eqnarray}
& &\left\Vert t\longmapsto \int_t^{T} S(u-t)\left(g_u^{n+1}- g_u^n\right)du\right\Vert_{L^p\left(\Omega;\gamma\left({T-\delta},T;E\right)\right)}\nonumber\\
& \leq & \delta \gamma(S)\left\Vert g^{n+1}- g^n\right\Vert_{L^p\left(\Omega;\gamma\left({T-\delta},T;E\right)\right)}\nonumber\\
& \leq & \delta \gamma(S)\left(L\left\Vert Y^n -Y^{n-1}\right\Vert_{L_{\mathbb{F}}^p\left(\Omega;\gamma\left({T-\delta},T;E\right)\right)}+L\left\Vert Z^n- Z^{n-1}\right\Vert_{L_{\mathbb{F}}^p\left(\Omega;\gamma\left({T-\delta},T;E\right)\right)}+\varepsilon_n\right).\label{ident329}
\end{eqnarray}
Therefore, taking into account (\ref{ident328}) and (\ref{ident329}) we derive from (\ref{ident318}) that:
\begin{eqnarray}
& &\left\Vert Y^{n+1}-Y^n\right\Vert_{L_{\mathbb{F}}^p\left(\Omega;\gamma\left({T-\delta},T;E\right)\right)}\label{ident330}\\
&\lesssim_{p,E}& \delta \gamma(S)(1+\gamma(S))\left(L\left\Vert Y^n -Y^{n-1}\right\Vert_{L_{\mathbb{F}}^p\left(\Omega;\gamma\left({T-\delta},T;E\right)\right)}+L\left\Vert Z^n- Z^{n-1}\right\Vert_{L_{\mathbb{F}}^p\left(\Omega;\gamma\left({T-\delta},T;E\right)\right)}+\varepsilon_n\right).\nonumber
\end{eqnarray}
Summing the two inequalities (\ref{ident324}) and (\ref{ident330}), member to member, to lead us to the following
\begin{eqnarray}
& &\left\Vert Y^{n+1}-Y^n\right\Vert_{L_{\mathbb{F}}^p\left(\Omega;\gamma\left({T-\delta},T;E\right)\right)}+\left\Vert Z^{n+1}-Z^n\right\Vert_{L_{\mathbb{F}}^p\left(\Omega;\gamma\left({T-\delta},T;E\right)\right)}\label{ident331}\\
&\leq &\beta\,\, \delta^{\frac{1}{2}}\left(\left\Vert Y^n -Y^{n-1}\right\Vert_{L_{\mathbb{F}}^p\left(\Omega;\gamma\left({T-\delta},T;E\right)\right)}+\left\Vert Z^n- Z^{n-1}\right\Vert_{L_{\mathbb{F}}^p\left(\Omega;\gamma\left({T-\delta},T;E\right)\right)}+\dfrac{\varepsilon_n}{L}\right),\nonumber
\end{eqnarray}
where
$$\beta:=c_{p,E} L \gamma(S)\left(1+T^{\frac{1}{2}} (1+\gamma(S))\right),$$ and $c_{p,E}$ is a constant depending on  both, the space $E$ and the real number $p$.

Iterating (\ref{ident331}) to obtain:
\begin{eqnarray*}
& &\left\Vert Y^{n+1}-Y^n\right\Vert_{L_{\mathbb{F}}^p\left(\Omega;\gamma\left({T-\delta},T;E\right)\right)}+\left\Vert Z^{n+1}-Z^n\right\Vert_{L_{\mathbb{F}}^p\left(\Omega;\gamma\left({T-\delta},T;E\right)\right)}\\
&\leq & (\beta \delta^{\frac12})^n\left(\left\Vert Y^1 -Y^{0}\right\Vert_{L_{\mathbb{F}}^p\left(\Omega;\gamma\left({T-\delta},T;E\right)\right)}+\left\Vert Z^1- Z^{0}\right\Vert_{L_{\mathbb{F}}^p\left(\Omega;\gamma\left({T-\delta},T;E\right)\right)}+\dfrac{\beta \delta^{\frac12}}{L}\displaystyle\sum_{\mathtt{i}=1}^{n}\dfrac{\varepsilon_{\mathtt{i}}}{(\beta \sqrt{\delta})^{\mathtt{i}}} \right).
\end{eqnarray*}
By choosing  $\delta = \dfrac{1}{4}\Big(\dfrac{1}{\beta^2}\wedge T\Big)$ and  $\varepsilon_{\mathtt{i}} =\Big(\dfrac{\beta \sqrt{\delta}}{2}\Big)^{\mathtt{i}}$, one can see that there exists a constant $C$ not depending on $n$ such that:
\begin{eqnarray*}\label{CCC}
	\left\Vert Y^{n+1}-Y^n\right\Vert_{L_{\mathbb{F}}^p\left(\Omega;\gamma\left({T-\delta},T;E\right)\right)}+\left\Vert Z^{n+1}-Z^n\right\Vert_{L_{\mathbb{F}}^p\left(\Omega;\gamma\left({T-\delta},T;E\right)\right)}
	\leq \dfrac{C}{2^n}.
\end{eqnarray*}
Henceforth,  $(Y^n)_n$ and $(Z^n)_n$ are Cauchy sequences in the Banach space $L_{\mathbb{F}}^p\left(\Omega;\gamma\left({T-\delta},T;E\right)\right)$. Thus by Theorem \ref{thm251} and the Brownian martingale representation theorem [\cite{BSEE.B}, (3.2)], they converge to $\mathscr{B}([T-\delta,T])\otimes\mathscr{F}$-measurable and adapted limits $Y$ and $Z$ respectively.\\
Moreover, by virtue of (\ref{ident205}) we have for any $x^*\in E^*$
\begin{eqnarray*}
\left\Vert <Y-Y^n,x^*>\right\Vert_{L^p(\Omega;L^2([T-\delta,T]))}\leq \left\Vert Y-Y^n\right\Vert_{L_{\mathbb{F}}^p\left(\Omega;\gamma\left({T-\delta},T;E\right)\right)}\left\Vert x^*\right\Vert_{E^*}.
\end{eqnarray*}
Then, $(<Y^n,x^*>)_n$ converges in measure to $<Y,x^*>$ for every $x^*\in E^*$.


\noindent {\textbf{Third step:}}
Now, we shall verify that the pair $(Y,Z)$ obtained in the second step leads us to solve locally the backward stochastic evolution inclusion (\ref{E1}) for $t\in [{T-\delta},T]$.\\
Since,
\begin{eqnarray*}
& &\left\Vert g^{n+1}-g^n\right\Vert_{L_{\mathbb{F}}^p\left(\Omega;\gamma\left({T-\delta},T;E\right)\right)}\\
&\leq & \rho\left(\Lambda_{{T-\delta},T}^{Y^n,Z^n},\Lambda_{{T-\delta},T}^{Y^{n-1},Z^{n-1}}\right)+\varepsilon_n\\
& \leq & L \left\Vert Y^n -Y^{n-1} \right\Vert_{L_{\mathbb{F}}^p\left(\Omega;\gamma\left({T-\delta},T;E\right)\right)}+L \left\Vert  Z^n - Z^{n-1}\right\Vert_{L_{\mathbb{F}}^p\left(\Omega;\gamma\left({T-\delta},T;E\right)\right)}+\varepsilon_n
\\ &\leq & L \dfrac{C}{2^{n-1}} +\Big(\dfrac{\beta \sqrt{\delta}}{2}\Big)^{n} \\ &\leq &\dfrac{LC +1}{2^{n}}
\end{eqnarray*}
then the sequence $\left(g^n\right)_n$ converges in $L_{\mathbb{F}}^p\left(\Omega;\gamma\left({T-\delta},T;E\right)\right)$.\\
According to hypothesis $(H_3)-(ii)$ we have for any natural number $n$
$$ \rho\left( \Lambda_{{T-\delta},T}^{Y,Z},\Lambda_{{T-\delta},T}^{Y^{n-1},Z^{n-1}}\right)\leq L \left\Vert Y -Y^{n-1}\right\Vert_{L_{\mathbb{F}}^p\left(\Omega;\gamma\left({T-\delta},T;E\right)\right)}+L \left\Vert Z- Z^{n-1}\right\Vert_{L_{\mathbb{F}}^p\left(\Omega;\gamma\left({T-\delta},T;E\right)\right)}.$$
Thus,
\begin{eqnarray}
\lim_{n\longrightarrow +\infty}d_{L_{\mathbb{F}}^p\left(\Omega;\gamma\left({T-\delta},T;E\right)\right)}\left( g^n,\Lambda_{{T-\delta},T}^{Y,Z}\right)=0,\label{ident333}
\end{eqnarray}
where $d_{L_{\mathbb{F}}^p\left(\Omega;\gamma\left(T-\delta,T;E\right)\right)}\left( g^{n},\Lambda_{T-\delta,T}^{Y,Z}\right)$ denotes the distance between the element of $L_{\mathbb{F}}^p\left(\Omega;\gamma\left(T-\delta,T;E\right)\right)$ defined by $g^{n}$ and the subset $\Lambda_{T-\delta,T}^{Y,Z}$.\\
It yields that the limit of $\left(g^n\right)_n$ belongs to the closure of $\Lambda_{{T-\delta},T}^{Y,Z}$ in $L_{\mathbb{F}}^p\left(\Omega;\gamma\left({T-\delta},T;E\right)\right)$. By virtue of hypothesis $(H_4)$ and Lemma \ref{lem301}, there exists a measurable process $g\in L_{\mathbb{F}}^p\left(\Omega;\gamma\left({T-\delta},T;E\right)\right)$ such that
$$
{g(u,\omega)}\in G(u, Y(u,\omega),Z(u,\omega))\,\, \text{for a.e.} (u,\omega) \in [T-\delta, T]\times \Omega.
$$

It remains to focus our analysis to the equation (\ref{ident317}) in order to show the aim of this step.
\\Let $t\in [{T-\delta},T]$ . We have:
\begin{eqnarray*}
\left\Vert \int_t^{T} S(u-t)\Big(g_u-g_u^n\Big)du\right\Vert_{L^p(\Omega;E)}\leq  \sqrt{T}\gamma(S)\left\Vert g-g^n\right\Vert_{L^p\left(\Omega;\gamma\left({T-\delta},T;E\right)\right)}
\end{eqnarray*}
and,
\begin{eqnarray*}
\left\Vert \int_t^{T} S(u-t)\Big(Z_u -Z_u^n\Big) dW_u\right\Vert_{L^p(\Omega;E)}
&\eqsim_{p,E}&\left\Vert S(.-t)(Z-Z^n)\right\Vert_{L^p\left(\Omega;\gamma\left(t,T;E\right)\right)}\\
&\leq & \gamma(S)\left\Vert Z-Z^n\right\Vert_{L^p\left(\Omega;\gamma\left({T-\delta},T;E\right)\right)}.
\end{eqnarray*}
Moreover, $Y^n$ is continuous in $L^p\left(\Omega;E\right)$ since the stochastic integral is continuous in pth moment due to Proposition $4.3$ in \cite{neerven3}, the second term in the left hand side of (\ref{ident317}) is continuous in $L^p\left(\Omega;E\right)$ according to Proposition \ref{prop251}, and the right hand side of (\ref{ident317}) is continuous in $L^p\left(\Omega;E\right)$ according to dominated convergence theorem. Therefore, $\left(Y^n\right)_{n\geq 1}$ converges in $C([T-\delta,T];L^p(\Omega;E))$ to
$$S(T-.)\xi-\displaystyle\int_.^{T} S(u-.)g_u du-\int_.^{T} S(u-.) Z_udW_u.$$ 
Since $E$ is separable reflexive then the dual space $E^*$ is separable, and due to Corollary $1.1.25$ in \cite{neerven1} the sequence $\left(Y^n\right)_{n\geq 1}$ converges in $C([T-\delta,T];L^p(\Omega;E))$ to a limit equal to $Y$ for a.e.$(u,\omega)$ .\\
Consequently, there exists a pair $(Y, Z)$ mild $L^p$-solution of the problem (\ref{E1}) on $[{T-\delta},T]$.

\noindent {\textbf{Fourth step:}}
For fear of not being complete we will see in this last step how to build a mild $L^p$-solution of the BSEI (\ref{E1}) by investing the existence of a local mild $L^p$-solution over a sufficiently small diameter interval included in $[0,T]$.

Let  $\delta$ is the value generated the convergence and the mild $L^p$-solution of (\ref{E1}) on $[{T-\delta},T]$ in the previous step. In an identical way, we have the existence of the solution on the interval
$[T-2\delta, T -\delta]$. By a finite number of steps, we cover in this way the whole interval $[0, T]$.
\end{proof}
\begin{rk}
It would be necessary to discuss the special case where, $h:[0,T]\times \Omega\times E\times E\longrightarrow E$ is a $\mathscr{B}([0,T])\otimes \mathscr{F}\otimes\mathscr{B}(E)\otimes\mathscr{B}(E)$-measurable mapping such that for any $(t,\omega, x,y)$ from $[0,T]\times \Omega\times E \times E$ $$G(t,x,y)=\Big\{ h(t,x,y)\Big\}.$$ Assume that the hypotheses $(H_1)-(H_3)$ are satisfied.\\
Let $Y,Y',Z$ and $Z'$ be measurable processes defining elements of $L_{\mathbb{F}}^p\left(\Omega;\gamma\left(0,T;E\right)\right)$.\\
Since $\Lambda^{Y,Z}\neq \emptyset$, then the process $h(.,Y,Z):(t,\omega)\longmapsto h(t,\omega,Y(t,\omega),Z(t,\omega))$ defines an element of $L_{\mathbb{F}}^p\left(\Omega;\gamma\left(0,T;E\right)\right)$. Further, we have from $(H_3)-(ii)$ the following:
\begin{eqnarray*}
\left\Vert h(.,Y,Z)-h(.,Y',Z')\right\Vert_{L_{\mathbb{F}}^p\left(\Omega;\gamma\left(0,T;E\right)\right)}&=& d_{L_{\mathbb{F}}^p\left(\Omega;\gamma\left(0,T;E\right)\right)}(h(.,Y,Z),\Lambda^{Y',Z'})\\
&=& \rho(\Lambda^{Y,Z},\Lambda^{Y',Z'})\\
&\leq & L\left(\left\Vert Y- Y'\right\Vert_{L_{\mathbb{F}}^p\left(\Omega;\gamma\left(0,T;E\right)\right)}+\left\Vert Z- Z'\right\Vert_{L_{\mathbb{F}}^p\left(\Omega;\gamma\left(0,T;E\right)\right)}\right).
\end{eqnarray*}
Therefore, if in addition $E$ has the upper contraction property, we deduce according to the study that was done to solve the BSEE (\ref{BSEE3}) that there exist $Y^*,Z^*$ two measurable and adapted processes defining elements of $L_{\mathbb{F}}^p\left(\Omega;\gamma\left(0,T;E\right)\right)$ such that for any $t\in [0,T]$:
$$Y^*_t+\int_t^T S(u-t) h(u,Y^*(u),Z^*(u))du+\int_t^T S(u-t) Z^*(u)dW_u= S(T-t)\xi \;\textrm{ in }\;L^p\left(\Omega;E\right).$$
It is worth mentioning that the set $\Lambda^{Y,Z}$ contains a unique equivalence class of the space $L_{\mathbb{F}}^p\left(\Omega;\gamma\left(0,T;E\right)\right)$, thus it was well closed and the hypothesis $(H_4)$ was already satisfied.
\end{rk}

\par Before introducing other result, the following lemma illustrates that the geometric property cotype $2$ guarantees the reduction of the assumptions needed to solve the BSEI (\ref{E1}).

\begin{lem}\label{lem302}
  Let $F$ be a separable reflexive Banach space with cotype $2$.
  \\Let $H: [0,T]\times \Omega\longrightarrow\mathscr{K}_{cmpt}(F)$ be a set-valued function and let $t',t''$ be two numbers from $[0,T]$, $t'<t'',$ such that the following set
  \begin{eqnarray}\label{ident339}
  \Delta_{t',t''}:=\left\lbrace h\in L_{\mathbb{F}}^p\left(\Omega;\gamma\left(t',t'';F\right)\right):h(t,\omega)\in H(t,\omega)\,a.e.(t,\omega)\in [t',t'']\times \Omega\right\rbrace
  \end{eqnarray} 
is non-empty. Then, $\Delta_{t',t''}$ is closed for the strong topology of $L_{\mathbb{F}}^p\left(\Omega;\gamma\left(t',t'';F\right)\right)$.
\end{lem}

\begin{proof}
	Let $\mathfrak{h}:[t',t''] \times \Omega\longrightarrow F$ be a measurable mapping defining an element of $L^p\left(\Omega;\gamma\left(t',t'';F\right)\right)$.
	\\Taking into account Theorem $9.2.11$ in \cite{neerven2} and H\"older inequality, we derive that:
	\begin{eqnarray}
		\upsilon_{2,F}^{\gamma}\left(\int_{\Omega} \left\Vert \mathfrak{h}\right\Vert_{\gamma(t',t'';F)}^p d\mathbb{P}\right)^{\frac{1}{p}}& \geq & \left(\int_{\Omega} \left\Vert \mathfrak{h}\right\Vert_{L^2([t',t''];F)}^p d\mathbb{P}\right)^{\frac{1}{p}}\nonumber\\
		& \geq & \left(\int_{\Omega} (t''-t')^{p(\frac{1}{2}-\frac{1}{2\wedge p})}\left\Vert \mathfrak{h}\right\Vert_{L^{2\wedge p}([t',t''];F)}^p d\mathbb{P}\right)^{\frac{1}{p}}\nonumber\\
		&\geq & (t''-t')^{\frac{1}{2}-\frac{1}{2\wedge p}}\left(\int_{\Omega}\left\Vert \mathfrak{h}\right\Vert_{L^{2\wedge p}([t',t''];F)}^{2\wedge p} d\mathbb{P}\right)^{\frac{1}{2\wedge p}},\label{ident304}
	\end{eqnarray}
	where $\upsilon_{2,F}^{\gamma}$ is the Gaussian cotype $2$ constant (\cite{neerven2}, Definition $7.1.17$).
	\\Thus, from (\ref{ident304}) and Tonelli's theorem we get to:
	\begin{eqnarray}\label{ident305}
		\upsilon_{2,F}^{\gamma}\left\Vert \mathfrak{h}\right\Vert_{L^p\left(\Omega;\gamma\left(t',t'';F\right)\right)}& \geq   		& (t''-t')^{\frac{1}{2}-\frac{1}{2\wedge p}}\left\Vert \mathfrak{h}\right\Vert_{L^{2\wedge p}([t',t'']\times			\Omega;F)}.
	\end{eqnarray}
	Let $(\mathfrak{h}^n)_{n\geq 1}$ be a sequence from $\Delta_{t',t''}$ converging to $\mathfrak{B}$ w.r.t. the 		strong topology of $L_{\mathbb{F}}^p\left(\Omega;\gamma\left(t',t'';F\right)\right)$.\\
According to (\ref{ident305}) one can see that $(\mathfrak{h}^n)_{n\geq 1}$ is a Cauchy sequence of $L^{2\wedge p}		([t',t'']\times\Omega;F)$, and thus it converges in this space to a measurable mapping $\mathfrak{h}:[t',t'']			\times\Omega\longrightarrow F$.
	\\Then there exists a subsequence $(\mathfrak{h}^{\varphi(n)})_{n\geq 1}$ such that:
	\begin{eqnarray*}\label{ident306}
		\mathfrak{h}^{\varphi(n)}\xrightarrow[n\to +\infty]{F}\mathfrak{h}\;\textrm{ for almost everywhere } (t,\omega)\in [t',t'']\times\Omega.
	\end{eqnarray*}
Keeping in mind the nature of the range of $H$, we infer that $\mathfrak{h}(t,\omega)\in H(t,\omega)$ for $a.e.(t,\omega)\in [t',t'']\times\Omega$.\\
Since $(\mathfrak{h}^{\varphi(n)})_{n\geq 1}$ is a sequence from $\Delta_{t',t''}$ converging to $\mathfrak{B}$, 		then there exists a further subsequence $(\mathfrak{h}^{\varphi\circ\phi(n)})_{n\geq 1}$ such that for almost surely $\omega\in \Omega$ we have:
	\begin{eqnarray*}\label{ident307}
		\mathfrak{h}^{\varphi\circ\phi(n)}(.,\omega)\xrightarrow[n\to +\infty]{\gamma\left(t',t'';F\right)}\mathfrak{B}(\omega).
	\end{eqnarray*}
	Let $\mathscr{N}'(\Omega)$ be the $\mathbb{P}$-negligible set which correspond to the previous assertion.\\ Let $\omega$ be a fixed element from the complement of $\mathscr{N}'(\Omega)$.
	By applying once more Theorem $9.2.11$ in \cite{neerven2}, we derive that for any natural numbers $n,m$ :
		\begin{eqnarray*}\label{ident308}
			\left\Vert \mathfrak{h}^{\varphi\circ\phi(n)}(.,\omega)-\mathfrak{h}^{\varphi\circ\phi(m)}(.,\omega)\right\Vert_{L^2([t',t''];F)}\leq \upsilon_{2,F}^{\gamma}\left\Vert \mathscr{I}_{\mathfrak{h}^{\varphi\circ\phi(n)}(.,\omega)}-\mathscr{I}_{\mathfrak{h}^{\varphi\circ\phi(m)}(.,\omega)}\right\Vert_{\gamma(t',t'';F)},
		\end{eqnarray*}
		where $\mathscr{I}_{\mathfrak{h}^{\varphi\circ\phi(n)}(.,\omega)}$ designes the Pettis integral operator associated to $\mathfrak{h}^{\varphi\circ\phi(n)}(.,\omega)$ in the sense of Definition \ref{def241}.
	Thus, $(\mathfrak{h}^{\varphi\circ\phi(n)}(.,\omega))_{n\geq 1}$ is a Cauchy sequence in $L^2([t',t''];F)$ for any $\omega$ outside of $\mathscr{N}'(\Omega)$. On the other hand, according to Tonelli's theorem we have
	$$
			\lim_{n\longrightarrow \infty}\left\Vert \mathfrak{h}^{\varphi\circ\phi(n)}-\mathfrak{h}\right\Vert_{L^{2\wedge p}([t',t'']\times\Omega;F)}=0 \implies \lim_{n\longrightarrow \infty}\int_{\Omega}\left\Vert \mathfrak{h}^{\varphi\circ\phi(n)}(.,\omega)-\mathfrak{h}(.,\omega)\right\Vert_{L^{2\wedge p}([t',t''];F)}^{2\wedge p}\mathbb{P}(d\omega)=0.
$$
Then, there exists a $\mathbb{P}$-negligible set $\mathscr{N}(\Omega)$ which contains $\mathscr{N}'(\Omega)$ and such that $(\mathfrak{h}^{\varphi\circ\phi(n)}(.,\omega))_{n\geq 1}$ converges to $\mathfrak{h}(.,\omega)$ with respect to the strong topology of $L^{2\wedge p}([t',t''];F)$ for any $\omega$ outside of $\mathscr{N}(\Omega)$. It follows that $(\mathfrak{h}^{\varphi\circ\phi(n)}(.,\omega))_{n\geq 1}$ converges to $\mathfrak{h}(.,\omega)$ with respect to the strong topology of $L^2([t',t''];F)$ and $\mathfrak{h}(.,\omega)\in L^2([t',t''];F)$ for any $\omega$ outside of $\mathscr{N}(\Omega)$. Subsequently, for any $\omega$ from $\Omega\setminus\mathscr{N}(\Omega)$, there exists a subsequence $(\mathfrak{h}^{\varphi\circ\phi\circ\psi_{\omega}(n)})_{n\geq 1}$ such that:
		\begin{eqnarray*}
			\mathfrak{h}^{\varphi\circ\phi\circ\psi_{\omega}(n)}(.,\omega)&\xrightarrow[n\to +\infty]{\gamma\left(t',t'';F\right)}&\mathfrak{B}(\omega),\\
			\mathfrak{h}^{\varphi\circ\phi\circ\psi_{\omega}(n)}(t,\omega)&\xrightarrow[n\to +\infty]{F}&\mathfrak{h}(t,\omega)\quad\textrm{for a.e.}\;t\in[t',t''].
		\end{eqnarray*}
		Then, according to Lemma $9.2.7$ in \cite{neerven2}, we infer that $\mathfrak{B}(\omega)=\mathscr{I}_{\mathfrak{h}(.,\omega)}$ for almost surely $\omega$. Moreover, $\mathfrak{h}$ defines the random variable $\mathfrak{B}$ by virtue of Lemma $2.7$ in \cite{neerven3}. Therefore $\mathfrak{h}$ defines an element of $\Delta_{t',t''}$. Finally $\Delta_{t',t''}$ is closed.
\end{proof}
\begin{examp}
Let $\kappa:[0,T]\times \Omega\longrightarrow \mathscr{K}_{cmpt}(\mathbb{R})$ be a $\mathscr{B}([0,T])\otimes\mathscr{F}$-measurable $\mathbb{F}$-adapted set-valued such that the quantity $\mathbb{E}\left(\left\Vert \vert\kappa\vert_{\varrho'}\right\Vert_{L^2([0,T])}^p\right)$ is finite, where $\varrho'$ is the Hausdorff distance on $\mathscr{K}_{cmpt}(\mathbb{R})$.\\
Let $x$ be a fixed element of $E$ and putting $G(u,\omega,.,.)=\kappa(u,\omega)x$ for any $(u,\omega)$ from $[0,T]\times \Omega$.\\
According to Proposition $2.5$ in \cite{jinping}, there is a $\Sigma_{\mathbb{F}}$-measurable selection $f$ of $\kappa$, and by virtue of Corollary $1.1.2$ in \cite{neerven1}, we infer that the function $f\otimes x$ is also $\Sigma_{\mathbb{F}}$-measurable. Moreover,
$$\left\Vert f\otimes x\right\Vert^p_{L^p\left(\Omega;\gamma\left(0,T;E\right)\right)}=\mathbb{E}\left(\left\Vert f\otimes x\right\Vert_{\gamma(0,T;E)}^p\right)=\mathbb{E}\left(\left\Vert f\right\Vert_{L^2([0,T])}^p\left\Vert x\right\Vert_{E}^p\right)\leq\left\Vert x\right\Vert_{E}^p\mathbb{E}\left(\left\Vert \vert\kappa\vert_{\varrho'}\right\Vert_{L^2([0,T])}^p\right),$$
thus the hypothesis $(H_3)$ is fulfilled.
\\Since $G$ takes images in $\mathscr{K}_{cmpt}(vect<x>)$, and $vect<x>$ has cotype $2$ according to Theorem $7.3.1$ in \cite{neerven2}, then $(H_4)$ is also fulfilled.
\end{examp}

 The second main theorem of this section is the following.
\begin{thm}\label{thm302}
Let $(H_1)-(H_3)$ be satisfied and assume in addition that $E$ has the upper contraction property and with cotype $2$. Then, the backward stochastic evolution inclusion (\ref{E1}) admits a mild $L^p$-solution $(Y,Z)$, and the associated process $g$ can be chosen $\Sigma_{\mathbb{F}}$-measurable.
\end{thm}
\begin{proof}
It suffices to arguing similarly as in the proof of Theorem \ref{thm301} and applying Lemma \ref{lem302} when calculating the limit of the sequence $(g_n)_n$ constructed in the settings (\ref{eq1}), in order to obtain a mild-$L^p$ solution $(Y,Z)$ with an associated measurable process $\tilde{g}\in \Lambda^{Y,Z}$.\\
Let $(\tilde{g}^n)_n$ be a sequence of $\mathbb{F}$-adapted step processes converging to $\tilde{g}$ in $L^p(\Omega;\gamma(0,T;E))$.
\\Similarly to (\ref{ident305}), we have for every natural numbers $n,m$
\begin{eqnarray}
\upsilon_{2,E}^{\gamma}\left\Vert \tilde{g}^n-\tilde{g}^m\right\Vert_{L^p\left(\Omega;\gamma\left(0,T;E\right)\right)}& \geq   		& T^{\frac{1}{2}-\frac{1}{2\wedge p}}\left\Vert \tilde{g}^n-\tilde{g}^m\right\Vert_{L^{2\wedge p}([0,T]\times\Omega,\Sigma_{\mathbb{F}};E)},\label{ident337}\\
\upsilon_{2,E}^{\gamma}\left\Vert \tilde{g}^n-\tilde{g}\right\Vert_{L^p\left(\Omega;\gamma\left(0,T;E\right)\right)}& \geq   		& T^{\frac{1}{2}-\frac{1}{2\wedge p}}\left\Vert \tilde{g}^n-\tilde{g}\right\Vert_{L^{2\wedge p}([0,T]\times\Omega,\mathscr{B}([0,T])\otimes\mathscr{F};E)}\label{ident338}.
\end{eqnarray}
From (\ref{ident337}), there exist a measurable and adapted process $g$, and a negligible set $\mathscr{N}_1\in \Sigma_{\mathbb{F}}$ such that $(\tilde{g}^n)_n$ converges pointwise to $g$ outside of $\mathscr{N}_1$. Moreover, (\ref{ident338}) asserts that there exists $\mathscr{N}_2\in \mathscr{B}([0,T])\otimes\mathscr{F}$ such that $(\tilde{g}^n)_n$ converges pointwise to $\tilde{g}$ outside of $\mathscr{N}_2$. Therefore $g$ and $\tilde{g}$ are equal outside of a negligible set $\mathscr{N}\in \mathscr{B}([0,T])\otimes\mathscr{F}$. Finally, $g$ and $\tilde{g}$ define the same element of $\Lambda^{Y,Z}$.
\end{proof}

The remaining lemma is concerned to determine, under suitable conditions, a relationship between the notion of subtrajectory integrals and the set defined by (\ref{ident339}).
\begin{lem}\label{lem303}
Let $F$ be a separable reflexive Banach space. Let  $H: [0,T]\times \Omega\longrightarrow\mathscr{K}_{cmpt}(F)$ be a set-valued function, let $p=2$ and assume that $\Delta_{0,T}$ is non-empty. The following assertions hold:
\begin{itemize}
\item[(a)]If $F$ has cotype $2$, then each element of $\Delta_{0,T}$ can be defined by an element of $S^2_{\Sigma_{\mathbb{F}}}(H)$.
\item[(b)]If $F$ is a Hilbert space, then each element $g\in S^2_{\Sigma_{\mathbb{F}}}(H)$ defines an element $U_g\in \Delta_{0,T}$ and we have $\left\Vert g\right\Vert_{L^2([0,T]\times \Omega;F)}=\left\Vert U_g\right\Vert_{L_{\mathbb{F}}^2(\Omega;\gamma(0,T;F))}$.
\end{itemize}
\end{lem}
\begin{proof}
$(a)$ Let $g$ be a measurable mapping defining an element X in $\Delta_{0,T}$. Since $F$ has cotype $2$, then it suffices to argue similarly as in the proof of Theorem \ref{thm302} in order to have the existence of a measurable and adapted mapping $\tilde{g}$ defining the same random variable $X$. Next, by using the continuous embedding $\gamma(0,T;F)\hookrightarrow L^2([0,T];F)$ we get $\tilde{g}\in S^2_{\Sigma_{\mathbb{F}}}(H)$.\\
$(b)$ Let $g$ be a $\Sigma_{\mathbb{F}}$-measurable mapping from $S^2_{\Sigma_{\mathbb{F}}}(H)$. Taking into account Proposition $1.2.24$ in \cite{neerven1}, the function $f:u\in [0,T]\longmapsto g(u,.)$ defines an element of $L^2([0,T];L^2(\Omega;F))$. Since $L^2(\Omega;F)$ is a Hilbert space, we derive from [\cite{neerven2}, Proposition $9.2.9$] that $L^2([0,T];L^2(\Omega;F))=\gamma(0,T;L^2(\Omega;F))$ isometrically. Thus $f$ belongs to $\gamma(0,T;L^2(\Omega;F))$.
\\Let $\mathscr{I}_f:L^2([0,T])\longrightarrow L^2(\Omega;F)$ the Pettis integrale operator in the sense of Definition \ref{def241}, and let $U$ be the random variable such that $\mathcal{I}_{\gamma}(U)=\mathscr{I}_f$, where $\mathcal{I}_{\gamma}$ is the isomorphism defined as in (\ref{ident207}) (with $s=0,\,t=T$ and  $E=F$). \\
Let $h\in L^2([0,T])$.
Since $\tilde{f}: u\in [0,T]\longmapsto h(u)f(u)$ defines an element of $L^1([0,T];L^2(\Omega;F))$, then the image $\mathscr{I}_f(h)$ and the Bochner integral of $\tilde{f}$ agree.\\
Let $(g^n)_{n\geq 1}$ be a sequence converging to $g$ in the space $L^2([0,T]\times \Omega,\mathscr{B}([0,T])\otimes\mathscr{F};F)$ such that $g^n=\sum_{j=1}^{k_n}\mathds{1}_{A_n^j}\mathds{1}_{B_n^j}x_j^n,\,A_n^j\in \mathscr{B}([0,T]),\,B_n^j\in \mathscr{F},\,x_j^n\in F$. It is easily to check that
\begin{eqnarray}\label{ident341}
\left(\int_{[0,T]}^{Bochner}h(u)g^n(u,.)du\right)(\omega)=\int_{[0,T]}^{Bochner}h(u)g^n(u,\omega)du,
\end{eqnarray}
for every pair $(n,\omega)$. Since for a subsequence, the right-hand side (resp. left-hand side) of (\ref{ident341}) converges pointwise for almost surely $\omega$, then
\begin{eqnarray}\label{ident342}
\left(\int_{[0,T]}^{Bochner}h(u)g(u,.)du\right)(\omega)=\int_{[0,T]}^{Bochner}h(u)g(u,\omega)du\quad a.s. \,  \omega.
\end{eqnarray}
On the other hand, by virtue of the isometry $\gamma(0,T;F)=L^2([0,T];F)$ we have $g(.,\omega)\in \gamma(0,T;F)$ for $a.s.\omega$. Thus, the Pettis integral $\mathscr{I}_{g(.,\omega)}(h)$ and the right hand side of (\ref{ident342}) agree for $a.s.\omega$, moreover $U(\omega)h=\mathscr{I}_{g(.,\omega)}(h)$ for $a.s.\omega$. Next, according to the separability of the space $L^2([0,T])$, we infer that there exists a $\mathbb{P}$-negligible set $\mathscr{N}$ such that
\begin{eqnarray*}
\forall \omega\in \Omega\setminus\mathscr{N},\;\forall h\in L^2([0,T]),\;U(\omega)h=\mathscr{I}_{g(.,\omega)}(h).
\end{eqnarray*}
Therefore $g(.,\omega)$ defines $U(\omega)$ for all $\omega\in \Omega\setminus\mathscr{N}$, and $g$ defines $U$.\\
Finally, it suffices to apply [\cite{neerven3}, Propositions $2.11$ and $2.12$] in order to accomplish the proof.
\end{proof}

\section{Martingale type $2$ Spaces's case}
Let $(\Omega, \mathscr{F}, \mathbb{F}, \mathbb{P})$ be a stochastic basis on which is defined a Brownian motion $(W_t)_{0\leq t\leq T}$ such that $\mathscr{F}=\mathscr{F}_T$ and $\mathbb{F}:=\Big(\mathscr{F}_t\big)_{t\in [0,T]}$ is the augmented natural filtration of $(W_t)_{0\leq t\leq T}$.\\

We start by declaring the following lemma which will prove both, measurability and adaptness of the orthogonal projection, of a measurable and adapted process onto a measurable and adapted set-valued under appropriate conditions.
\begin{lem}\label{lem402}
Let $F$ be a reflexive separable Banach space.

Let $\Theta:[0,T]\times \Omega\rightarrow \mathscr{K}_{ccmpt}\left(F\right)$ and $\varphi:[0,T]\times \Omega\rightarrow F$ be two $\mathscr{B}([0,T])\otimes \mathscr{F}$-measurable and adapted processes.

Then, there exists a measurable and adapted selection $\theta$ of $\Theta$ such that for any $(t,\omega)\in [0,T]\times \Omega$,\;we have:
$$d_{F}\left(\varphi(t,\omega),\Theta(t,\omega)\right)=\Vert \theta(t,\omega)-\varphi(t,\omega)\Vert_{F}.$$
\end{lem}
\begin{proof}
Let $(t,\omega)\in [0,T]\times\Omega$. Since $\Theta(t,\omega)$ is a nonempty closed convex set of $F$, and $F$ is reflexive, then there exists $\theta(t,\omega)$ an element of $\Theta(t,\omega)$ such that $d_{F}\left(\varphi(t,\omega),\Theta(t,\omega)\right)=\Vert \theta(t,\omega)-\varphi(t,\omega)\Vert_{F}.$\\
Next, since $\Theta$ is a $\mathscr{B}([0,T])\otimes \mathscr{F}$-measurable and adapted set-valued process, then it is $\Sigma_{\mathbb{F}}$-measurable. Thus according to Castaing representation theorem we derive that there exists a sequence $\left(\theta^n\right)_{n\geq 1}$ of measurable and adapted selections of $\Theta$ such that $\Theta\left(t,\omega\right)=cl_{F}(\cup_{n=1}^\infty\left\lbrace \theta^n\left(t,\omega\right)\right\rbrace)$ for all $(t,\omega)\in [0,T]\times \Omega$. It follows that:
\begin{eqnarray*}
\forall (t,\omega)\in [0,T]\times\Omega,\quad d_{F}\left(\varphi(t,\omega),\Theta(t,\omega)\right)=\inf_{n\geq 1}\left\Vert \varphi(t,\omega)-\theta^n\left(t,\omega\right)\right\Vert_{F}.
\end{eqnarray*}
Then, the process $(t,\omega)\in [0,T]\times\Omega\longmapsto d_{F}\left(\varphi(t,\omega),\Theta(t,\omega)\right)$ is measurable and adapted.\\
Taking into account Theorem III.$4$1 in \cite{closedball}, we infer that the set-valued:
\begin{eqnarray*}
(t,\omega)\in [0,T]\times\Omega\longmapsto B^f\left(\varphi(t,\omega),d_{F}\left(\varphi(t,\omega),\Theta(t,\omega)\right)\right)
\end{eqnarray*}
is $\mathscr{B}([0,T])\otimes \mathscr{F}$-measurable and adapted, where $B^f$ designates a closed ball in $F$. So, the set-valued
\begin{eqnarray*}
\left(t,\omega\right)\in [0,T]\times\Omega\longmapsto B^f\left(\varphi(t,\omega),d_{F}\left(\varphi(t,\omega),\Theta(t,\omega)\right)\right)\bigcap \Theta\left(t,\omega\right)
\end{eqnarray*}
is $\mathscr{B}([0,T])\otimes \mathscr{F}$-measurable and adapted by applying Proposition $6.5.7$ in \cite{intermeas}. On the other hand,
\begin{eqnarray*}
B^f\left(\varphi(t,\omega),d_{F}\left(\varphi(t,\omega),\Theta(t,\omega)\right)\right)\bigcap \Theta\left(t,\omega\right)=\left\lbrace \theta(t,\omega)\right\rbrace
\end{eqnarray*}
for any $\left(t,\omega\right)$ from $[0,T]\times \Omega$. Therefore, the mapping $\theta:[0,T]\times\Omega\longrightarrow F$ is measurable and adapted.
\end{proof}

Let $E$ be a separable UMD space. Let $A$ be the generator of a $C_0$-semigroup $\left\lbrace S_t\right\rbrace_{t\geq 0}$ on $E$. We now propose to study the following backward stochastic evolution inclusion:
\begin{eqnarray}\label{E2}
BSEI(2):\begin{cases}
dY_u+AY_udu &\in   \mathbf{G}\left(u,Y_u,Z_u\right)du +Z_udW_u\\
Y_T&= \xi
\end{cases}
\end{eqnarray}
where $\mathbf{G}:\left[0,T\right]\times\Omega\times E\times E\rightarrow \mathscr{K}_{ccmpt}\left(E\right)$ is a $\mathscr{B}([0,T])\otimes\mathscr{F}\otimes \mathscr{B}(E)\otimes \mathscr{B}(E)$-measurable set-valued.
It can be seen that the set-valued introduced in problem (\ref{E2}) has a more specific range compared to the set-valued in problem (\ref{E1}), which will enable us to exploit Lemma \ref{lem402}.\bigskip\\
Let us consider the following package of hypotheses:\\
\begin{itemize}
\item[$(H_1')$]
\begin{itemize}
\item[(i)] $\xi$ belongs to $L^p\left(\Omega,\mathscr{F}_T,\mathbb{P};E\right)$;
\item[(ii)]the set $\left\lbrace  S_t\right\rbrace_{t\in [0,T]}$ is $\gamma$-bounded.
\end{itemize}
\item[$(H_2')$]
The set-valued $\mathbf{G}$ has the following properties:
\begin{itemize}
\item[(i)]
$\mathbf{G}$ is a $\Sigma_{\mathbb{F}}\otimes \mathscr{B}(E)\otimes \mathscr{B}(E)$-measurable set-valued;
\item[(ii)]
$\mathbb{E}\left( \left\Vert\left\vert \mathbf{G}\left(.,0_{E},0_{E}\right)\right\vert_{\varrho}\right\Vert_{L^2([0,T])}^p\right)<\infty$;
\item[(iii)]for any measurable mappings $\mathtt{Y},\mathtt{Y}',\mathtt{Z},\mathtt{Z}':[0,T]\longrightarrow E$ defining elements in $\gamma(0,T;E)$:
\begin{eqnarray*}
\left\Vert \varrho\left(\mathbf{G}(.,\mathtt{Y},\mathtt{Z}),\mathbf{G}(.,\mathtt{Y}',\mathtt{Z}')\right)\right\Vert_{L^2([0,T])}\leq K\left(\left\Vert \mathtt{Y}- \mathtt{Y}'\right\Vert_{\gamma\left(0,T;E\right)}+\left\Vert \mathtt{Z}- \mathtt{Z}'\right\Vert_{\gamma\left(0,T;E\right)}\right).
\end{eqnarray*}
\end{itemize}
Where $\varrho$ is the Hausdorff distance on $\mathscr{K}_{cmpt}(E)$.
\end{itemize}
\begin{lem}\label{lem403}Let $s,t$ be two real numbers such that $0\leq s<t\leq T$.\\
If $(H'_2)$ is fulfilled then the following properties hold:
	\begin{itemize}
		\item[$1)$]
		For any measurable mappings $\mathtt{Y},\mathtt{Z}:[s,t]\longrightarrow E$ defining elements in $\gamma(s,t;E)$, we have
		\begin{eqnarray}
		\left\Vert \varrho\left(\mathbf{G}(.,\mathtt{Y},\mathtt{Z}),\mathbf{G}(.,\mathtt{Y}',\mathtt{Z}')\right)\right\Vert_{L^2([s,t])}\leq K\left(\left\Vert \mathtt{Y}- \mathtt{Y}'\right\Vert_{\gamma\left(s,t;E\right)}+\left\Vert \mathtt{Z}- \mathtt{Z}'\right\Vert_{\gamma\left(s,t;E\right)}\right),\label{ident403}
		\end{eqnarray}
		\item[$2)$]
		For any $\Sigma_{\mathbb{F}}^{s,t}$-measurable mappings $Y,\,Z:[s,t]\times \Omega\longrightarrow E$ defining elements in $L^p(\Omega;\gamma(s,t;E))$, the space $L^p(\Omega;L^2([s,t];E))$ contains all elements of $\mathbb{S}_{\mathbf{G}(.,Y,Z)}$,\\
		where $\Sigma_{\mathbb{F}}^{s,t}:=\left\lbrace A\in \mathscr{B}([s,t])\otimes\mathscr{F}:\mathds{1}_A(u,.)\textrm{ is }\mathscr{F}_u\textrm{-measurable for every }\,u\in [s,t]\right\rbrace$, and $\mathbb{S}_{\mathbf{G}(.,Y,Z)}$ is consisting of all $\Sigma_{\mathbb{F}}^{s,t}$-measurable selections of the set-valued $\mathbf{G}(.,Y,Z)$ defined on $[s,t]\times\Omega$.
	\end{itemize}
\end{lem}
\begin{proof}
$1)$ Let $\mathtt{Y},\mathtt{Z}:[s,t]\longrightarrow E$ be two measurable mappings defining elements in $\gamma(s,t;E)$. Since $\mathds{1}_{[s,t]}\mathtt{Y},\,\mathds{1}_{[s,t]}\mathtt{Z}\in \gamma(0,T;E)$, we derive successively
\begin{eqnarray*}
& & K^2\left(\left\Vert \mathtt{Y}- \mathtt{Y}'\right\Vert_{\gamma\left(s,t;E\right)}+\left\Vert \mathtt{Z}- \mathtt{Z}'\right\Vert_{\gamma\left(s,t;E\right)}\right)^2\\
&=& K^2\left(\left\Vert \mathds{1}_{[s,t]}\mathtt{Y}- \mathds{1}_{[s,t]}\mathtt{Y}'\right\Vert_{\gamma\left(0,T;E\right)}+\left\Vert \mathds{1}_{[s,t]}\mathtt{Z}- \mathds{1}_{[s,t]}\mathtt{Z}'\right\Vert_{\gamma\left(0,T;E\right)}\right)^2\\
& \geq &\left\Vert \varrho\left(\mathbf{G}(.,\mathds{1}_{[s,t]}\mathtt{Y},\mathds{1}_{[s,t]}\mathtt{Z}),\mathbf{G}(.,\mathds{1}_{[s,t]}\mathtt{Y}',\mathds{1}_{[s,t]}\mathtt{Z}')\right)\right\Vert^2_{L^2([0,T])}\\
&=&\int_{[0,T]}\left(\mathds{1}_{[s,t]}+\mathds{1}_{[0,T]\setminus [s,t]}\right)\varrho^2\left(\mathbf{G}(u,\mathds{1}_{[s,t]}(u)\mathtt{Y}(u),\mathds{1}_{[s,t]}(u)\mathtt{Z}(u)),\mathbf{G}(u,\mathds{1}_{[s,t]}(u)\mathtt{Y}'(u),\mathds{1}_{[s,t]}(u)\mathtt{Z}'(u))\right)du\\
&=&\int_{[s,t]}\varrho^2\left(\mathbf{G}(u,\mathtt{Y}(u),\mathtt{Z}(u)),\mathbf{G}(u,\mathtt{Y}'(u),\mathtt{Z}'(u))\right)du.
\end{eqnarray*}
Thus the first aim follows.\\
$2)$ Let $Y,\,Z:[s,t]\times \Omega\longrightarrow E$, be two $\Sigma_{\mathbb{F}}^{s,t}$-measurable mappings belonging to $L^p\left(\Omega;\gamma\left(s,t;E\right)\right)$.
\\By virtue of Lemma \ref{lem401}, we infer that $\mathbf{G}(.,Y,Z)$ is $\Sigma^{s,t}_{\mathbb{F}}$-measurable and $\mathbb{S}_{\mathbf{G}(.,Y,Z)}$ is non-empty.
\\Let $h$ be in $\mathbb{S}_{\mathbf{G}(.,Y,Z)}$. Due to Lemma \ref{lem402}, there is an element $\tilde{h}:[s,t]\times\Omega\rightarrow E$ from $\mathbb{S}_{\mathbf{G}(.,0_E,0_E)}$ such that for all $(u,\omega)\in [s,t]\times\Omega$ we have
$$d_E(h(u,\omega),\mathbf{G}(u,0_E,0_E))=\left\Vert h(u,\omega)-\tilde{h}(u,\omega)\right\Vert_E.$$ 
Keeping in mind (\ref{ident403}) we derive successively:
\begin{eqnarray*}
\left\Vert h\right\Vert_{L^p(\Omega;L^2([s,t];E))}&\leq & \left\Vert d_E(h,\mathbf{G}(.,0_E,0_E))\right\Vert_{L^p(\Omega;L^2([s,t]))}+\left\Vert \tilde{h}\right\Vert_{L^p(\Omega;L^2([s,t]))}\\
&\leq & \left\Vert \varrho(\mathbf{G}(.,Y,Z),\mathbf{G}(.,0_E,0_E))\right\Vert_{L^p(\Omega;L^2([s,t]))}+\left\Vert \vert \mathbf{G}(.,0_E,0_E)\vert_{\varrho}\right\Vert_{L^p(\Omega;L^2([s,t]))}\\
&\leq & K\left(\left\Vert Y\right\Vert_{L_{\mathbb{F}}^p\left(\Omega;\gamma\left(s,t;E\right)\right)}+\left\Vert Z\right\Vert_{L_{\mathbb{F}}^p\left(\Omega;\gamma\left(s,t;E\right)\right)}\right)+\left\Vert \vert \mathbf{G}(.,0_E,0_E)\vert_{\varrho}\right\Vert_{L^p(\Omega;L^2([s,t]))}.
\end{eqnarray*}
Therefore, the second aim follows.
\end{proof}
\begin{rks}\label{rk401}
\begin{itemize}
\item[$(a)$]
Let $(r_n')_{n\geq 1}$ and $(r_n'')_{n\geq 1}$ be sequences of independent Rademacher random variables on independent probability spaces $(\Omega',\mathscr{F}',\mathbb{P}')$ and $(\Omega'',\mathscr{F}'',\mathbb{P}'')$ respectively, and let $(r_{nm}''')_{n,m\geq 1}$ be a doubly indexed sequence of independent Rademacher random variables on a probability space $(\Omega''',\mathscr{F}''',\mathbb{P}''')$(see \cite{neerven1}, Definition 3.2.9).\\
By taking instead of the finite Gaussian sequences $(\gamma_i')_{1\leq i\leq l},\,(\gamma_j'')_{1\leq j\leq k},\,(\gamma_{ij}''')_{1\leq i\leq l \atop 1\leq j\leq k}$, the finite Rademacher sequences $(r_i')_{1\leq i\leq l},\,(r_j'')_{1\leq j\leq k},\,(r_{ij}''')_{1\leq i\leq l \atop 1\leq j\leq k}$ respectively, and replacing in (\ref{ident241}), we get an equivalent property if the space $E$ has finite cotype (see [\cite{neerven2}, Corollary $7.2.10$]).\\Subsequently, every Banach space with the geometric property type $2$ has the contraction property according to [\cite{neerven2}, Corollary $7.3.11$] and [\cite{neerven6}, Proposition $2.7$].
\item[$(b)$]Every martingale type $2$ space has type $2$, and for UMD spaces these two properties are equivalent (see \cite{neerven1}, Definition $3.5.23$, Proposition $4.3.13$).
\end{itemize}
\end{rks}

\par The first main result of this section where we will also count on Definition \ref{def301} of a mild $L^p$-solution, is the following.
\begin{thm}\label{thm401}
Assume the hypotheses $(H_1')-(H_2')$ are satisfied. If $E$ has martingale type $2$, then the backward stochastic evolution inclusion (\ref{E2}) admits a mild $L^p$-solution $(Y,Z)$, and the associated process $g$ can be chosen $\Sigma_{\mathbb{F}}$-measurable belonging to $L^p\left(\Omega;L^2\left([0, T];E\right)\right)$.
\end{thm}
\begin{proof}
Let $0<\delta \leq T$. By applying iteratively Lemma \ref{lem402} and Lemma \ref{lem403}, we can construct $\left(Y^n\right)_{n},\,\left(Z^n\right)_{n},\,\left(g^n\right)_{n}$ sequences of $\Sigma_{\mathbb{F}}^{T-\delta,T}$-measurable processes such that for any $n$ we have:
\begin{eqnarray}\label{eq2}
\left\{
	\begin{array}{lll}
		 (i) &
g^n \in \mathbb{S}_{\mathbf{G}(.,Y^{n-1},Z^{n-1})}\bigcap L^p\left(\Omega;L^2\left([T-\delta, T];E\right)\right)    \text{ such that }\\&
	\hspace*{2cm} d_E\left(g^{n-1},\mathbf{G}(.,Y^{n-1},Z^{n-1})\right)=\left\Vert g^{n}-g^{n-1}\right\Vert_{E} ;

	  \\ (ii) & 	Y^n\; (resp.\, Z^n)\text{ defines element of the space  }L_{\mathbb{F}}^p\left(\Omega;\gamma\left(T-\delta,T;E\right)\right) \text{ such that}  \\ &
	   Y_t^{n}+\displaystyle\int_t^{T} S(u-t) g_u^n du+\int_t^{T} S(u-t) Z_u^{n}dW_u=S(T-t)\xi\; \textrm{ in }L^p\left(\Omega;E\right).

	\end{array}
	\right.
\end{eqnarray}
By virtue of Theorem $9.2.10$ in \cite{neerven2} we derive that
\begin{eqnarray}\label{ident404}
& &\left\Vert g^{n+1}- g^{n}\right\Vert_{L_{\mathbb{F}}^p\left(\Omega;\gamma\left(T-\delta,T;E\right)\right)}\nonumber\\
&\leq & \varsigma_{2,E}^{\gamma}\left\Vert g^{n+1}- g^{n}\right\Vert_{L^p(\Omega;L^2([T-\delta,T];E))}\nonumber\\
&\leq &K \varsigma_{2,E}^{\gamma}\left(\left\Vert Y^n- Y^{n-1}\right\Vert_{L_{\mathbb{F}}^p\left(\Omega;\gamma\left(T-\delta,T;E\right)\right)}+\left\Vert Z^n- Z^{n-1}\right\Vert_{L_{\mathbb{F}}^p\left(\Omega;\gamma\left(T-\delta,T;E\right)\right)}\right)
\end{eqnarray}
where $\varsigma_{2,E}^{\gamma}$ is the Gaussian type $2$ constant (\cite{neerven2}, Definition $7.1.17$). \\Taking in mind Remarks \ref{rk401} and reasoning in the same way as in the proof of Theorem $\ref{thm301}$, we can choose $\delta$ in order to ensure that there is a positive constant $C$ independent on $n$ such that:
\begin{eqnarray}\label{ident405}
\left\Vert Y^{n+1}- Y^{n}\right\Vert_{L_{\mathbb{F}}^p\left(\Omega;\gamma\left(T-\delta,T;E\right)\right)}+\left\Vert Z^{n+1}- Z^{n}\right\Vert_{L_{\mathbb{F}}^p\left(\Omega;\gamma\left(T-\delta,T;E\right)\right)}\leq \dfrac{C}{2^n}.
\end{eqnarray}
By Combining (\ref{ident404}) and (\ref{ident405}) we infer that the sequence $(g^n)_n$ converges to a measurable and adapted process $g$ belonging to the space $L^p(\Omega;L^2([T-\delta,T];E))$ and defining element in $L^p\left(\Omega;\gamma\left(T-\delta,T;E\right)\right)$.\\
Next, due to Proposition \ref{prop252} there exists $\tau\in L_{\mathbb{F}}^p\left(\Omega;\gamma\left({T-\delta},T;\gamma\left({T-\delta},T;E\right)\right)\right)$ such that for almost all $u\in [{T-\delta},T]$ we have in $L^p(\Omega;E)$ the following equality
\begin{eqnarray*}
g_u &=&\mathbb{E}\left(g_u\right)+\int_{T-\delta}^u \tau(u,s)dW_s.
\end{eqnarray*}
By martingale representation in UMD spaces, there exists $\Psi\in L_{\mathbb{F}}^p(\Omega;\gamma(T-\delta,T;E))$ such that $\int_u^T\Psi_sdW_s=\xi-\mathbb{E}(\xi\vert \mathscr{F}_u)$.
Now, let's define a mapping $Z:[T-\delta,T]\times \Omega\longrightarrow E$ as follows
\begin{eqnarray*}
Z_u=S(T-u)\Psi_u+\int_u^T S(s-u)\tau(s,u)ds.
\end{eqnarray*}
Thus, similarly to the estimates (\ref{estimates}) we have
\begin{eqnarray*}
\left\Vert Z^{n}-Z\right\Vert_{L_{\mathbb{F}}^p\left(\Omega;\gamma\left({T-\delta},T;E\right)\right)} \lesssim_{p,E}  \delta^{\frac{1}{2}}\gamma(S)\left\Vert g^{n}-g\right\Vert_{L_{\mathbb{F}}^p\left(\Omega;\gamma\left({T-\delta},T;E\right)\right)}.
\end{eqnarray*}
Therefore, the sequence $(Z^n)_n$ converges in $L_{\mathbb{F}}^p\left(\Omega;\gamma\left(T-\delta,T;E\right)\right)$ to $Z$, and by a similar arguments as in the proof of [Theorem \ref{thm301}, $3^{th}$ step], the sequence $(Y^n)_n$ converges in both spaces $L_{\mathbb{F}}^p\left(\Omega;\gamma\left(T-\delta,T;E\right)\right)$ and $C([T-\delta,T],L^p(\Omega;E))$ to a measurable and adapted process $Y$.\\
Taking into account Lemma \ref{lem403} we derive that the sequence $\left(d_E(g^n,G(.,Y,Z))\right)_{n\geq 1}$ converges to $0$ for $a.e.(u,\omega)$. Then $g(u,\omega)\in G(u,Y(u,\omega),Z(u,\omega))$ for $a.e.(u,\omega)\in [T-\delta,T]\times \Omega$.\\
Consequently, the problem (\ref{E2}) possesses a mild $L^p$-solution $(Y,Z)$ on $[T-\delta,T]$, with an associated process $g$ measurable and adapted belonging to $L^p(\Omega;L^2([T-\delta,T];E))$.\\
What remains to complete the arguments is mentioned in the proof of Theorem \ref{thm301}.
\end{proof}

\par We are left with a last particular result, which is characterized by the specificity of the adopted space, weaken the assumptions, as well as the formulation of its hypothesis linked to the Lipschitz condition. To conclude, $(H_1'')$ will have the content of $(H_1)'-(i)$, further we launch another hypothesis $(H_2'')$, instead of $(H_2 )'$, which is obtained by modifying the part $(H_2')-(iii)$ while keeping $(H_2 ')-(i)$ and $(H_2')-(ii)$ unchanged. We take:\\
\par $(H_2'')-(iii)$\; For all $ y,y',z,z'\in E$ we have:
$$\varrho \left(\mathbf{G}\left(.,y,z\right);\mathbf{G}\left(.,y',z'\right)\right)\leq \tilde{K}\left(\left\Vert y-y'\right\Vert_{E}+\left\Vert z-z'\right\Vert_{E}\right).$$
\begin{cor}\label{cor401}
Assume the hypotheses $(H_1'')-(H_2'')$ are satisfied. If $E$ is a separable Hilbert space, then the backward stochastic evolution inclusion (\ref{E2}) admits a mild $L^p$-solution $(Y,Z)$, and the associated process $g$ can be chosen $\Sigma_{\mathbb{F}}$-measurable and belonging to $L^p(\Omega;L^2([0,T];E))$.
\end{cor}
\begin{proof}
Since E is a Hilbert space, then it is a UMD Banach space and has type $2$ according to Theorem $7.3.1$ in \cite{neerven2}. Moreover, the set $\lbrace S_t\rbrace_{t\in [0,T]}$ is uniformly bounded, thus it is also $\gamma$-bounded due to Theorem $8.1.3$ in \cite{neerven2}. Therefore, the claim is a straightforward result of Theorem $\ref{thm401}$ by taking into account the isometry $\gamma(0,T;E)=L^2([0,T];E)$.
\end{proof}

In the end, we determine a relationship between the set of assumptions used in the previous section and the set of the last one:
\begin{thm}
Let $G:\left[0,T\right]\times\Omega\times E\times E\rightarrow \mathscr{K}_{ccmpt}\left(E\right)$ be a $\Sigma_{\mathbb{F}}\otimes \mathscr{B}(E)\otimes \mathscr{B}(E)$-measurable set-valued, and let us consider the two following packages of assertions:
\begin{eqnarray*}
(\Gamma_1)\left\{
	\begin{array}{lll}
		 (i) &
\textrm{For all }\,Y,Z \textrm{ defining elements of } L_{\mathbb{F}}^2\left(\Omega;\gamma\left(0,T;E\right)\right),\;\textrm{the set}\\&
\Lambda^{Y,Z}:=\left\lbrace g\in L_{\mathbb{F}}^2\left(\Omega;\gamma\left(0, T;E\right)\right):g\in G(.,Y,Z)\,a.e.(u,\omega)\in [0, T]\times \Omega\right\rbrace\;\textrm{is non-empty};

	\\ (ii)& \sup_{g\in \Lambda^{0_E,0_E}}\left\Vert g\right\Vert_{L_{\mathbb{F}}^2\left(\Omega;\gamma\left(0,T;E\right)\right)}<\infty;

	  \\ (iii) & 	\textrm{There exists } K_1>0 \textrm{ s.t. for all }\,Y,Z,Y',Z' \textrm{ defining elements of }L_{\mathbb{F}}^2\left(\Omega;\gamma\left(0,T;E\right)\right):  \\ &
	   \rho(\Lambda^{Y,Z},\Lambda^{Y',Z'})\leq K_1\left(\left\Vert Y- Y'\right\Vert_{L_{\mathbb{F}}^2\left(\Omega;\gamma\left(0,T;E\right)\right)}+\left\Vert Z- Z'\right\Vert_{L_{\mathbb{F}}^2\left(\Omega;\gamma\left(0,T;E\right)\right)}\right),
	   \\& \textrm{where } \rho(\Lambda^{Y,Z},\Lambda^{Y',Z'}) \textrm{ is the Hausdorff distance between } \Lambda^{Y,Z}\textrm{ and }\Lambda^{Y',Z'}.

	\end{array}
	\right.
\end{eqnarray*}
\begin{eqnarray*}
(\Gamma_2)\left\{
	\begin{array}{lll}
		 (i) & \mathbb{E}\left( \left\Vert\left\vert G\left(.,0_{E},0_{E}\right)\right\vert_{\varrho}\right\Vert_{L^2([0,T])}^2\right)<\infty
;

	  \\ (ii) & 	\textrm{There exists } K_2>0  \textrm{ s.t. for almost everywhere}\, (u,\omega)\,\textrm{and for all }\,y,z,y',z' \textrm{ of } E:  \\ &
	   \varrho \left(G\left(u,y,z\right);G\left(u,y',z'\right)\right)\leq K_2\left(\left\Vert y-y'\right\Vert_{E}+\left\Vert z-z'\right\Vert_{E}\right).

	\end{array}
	\right.
\end{eqnarray*}
If $E$ is a Hilbert space, then $(\Gamma_1)$ and $(\Gamma_2)$ are equivalent.
\end{thm}
\begin{proof}
Assume that $(\Gamma_1)$ is fulfilled.\\
By virtue of Lemma \ref{lem303} and [\cite{set-stochastic-ch2}, Remark $2.3.2$] we infer
\begin{eqnarray}\label{ident406}
\sup_{g\in \Lambda^{0_E,0_E}}\left\Vert g\right\Vert^2_{L_{\mathbb{F}}^2\left(\Omega;\gamma\left(0,T;E\right)\right)}=\sup_{g\in S^2_{\Sigma_{\mathbb{F}}}(G(.,0_E,0_E))}\int_{[0,T]\times\Omega}\left\Vert g\right\Vert_E^2 dud\mathbb{P}=\int_{[0,T]\times\Omega}\sup_{x\in G(u,0_E,0_E)}\left\Vert x\right\Vert_E^2 dud\mathbb{P}.
\end{eqnarray}
Then $(\Gamma_2)-(i)$ is obtained.\bigskip\\
Let $\tilde{Y},\tilde{Z}$ be two measurable and adapted mappings defining elements of $L_{\mathbb{F}}^2\left(\Omega;\gamma\left(0,T;E\right)\right)$.\\
Keeping in mind Lemma $1.3.1$ in \cite{set-stochastic-ch1} we have
\begin{eqnarray*}
\left\vert\Lambda^{\tilde{Y},\tilde{Z}}\right\vert_{\rho} &\leq & \rho\left(\Lambda^{\tilde{Y},\tilde{Z}},\Lambda^{0_E,0_E}\right)+\rho\left(\Lambda^{0_E,0_E},\left\lbrace 0_{L_{\mathbb{F}}^2\left(\Omega;\gamma\left(0,T;E\right)\right)}\right\rbrace\right)\\
&\leq & K_1\left(\left\Vert \tilde{Y}\right\Vert_{L_{\mathbb{F}}^2\left(\Omega;\gamma\left(0,T;E\right)\right)}+\left\Vert \tilde{Z}\right\Vert_{L_{\mathbb{F}}^2\left(\Omega;\gamma\left(0,T;E\right)\right)}\right)+ \sup_{g\in \Lambda^{0_E,0_E}}\left\Vert g\right\Vert_{L_{\mathbb{F}}^2\left(\Omega;\gamma\left(0,T;E\right)\right)}.
\end{eqnarray*}
Thus, by taking into account once more Lemma \ref{lem303} we derive 
\begin{eqnarray}\label{ident407}
\sup_{g\in S^2_{\Sigma_{\mathbb{F}}}(G(.,\tilde{Y},\tilde{Z}))}\int_{[0,T]\times\Omega}\left\Vert g\right\Vert_E^2 dud\mathbb{P}<\infty.
\end{eqnarray}
Now, let's take $Y=\mathds{1}_A y,\;Y'=\mathds{1}_A y',\;Z=\mathds{1}_A z,\;Z'=\mathds{1}_A z',$ where $(y,y',z,z')\in E^4$ and $A$ be a non-empty $\Sigma_{\mathbb{F}}$-measurable set. According to (\ref{ident407}) and [\cite{set-stochastic-ch2}, Theorem $2.3.4$] we have successively
\begin{eqnarray*}
\sup_{g\in \Lambda^{Y,Z}}\inf_{g'\in \Lambda^{Y',Z'}}\left\Vert g-g'\right\Vert^2_{L_{\mathbb{F}}^p\left(\Omega;\gamma(0,T;E)\right)}
&=& \sup_{g\in S^2_{\Sigma_{\mathbb{F}}}(G(.,Y,Z))}\inf_{g'\in S^2_{\Sigma_{\mathbb{F}}}(G(.,Y',Z'))}\left\Vert g-g'\right\Vert^2_{L^2\left([0,T]\times\Omega,\Sigma_{\mathbb{F}};E\right)}\\
&=&\int_{[0,T]\times\Omega}\sup_{x\in G(u,Y,Z)}\inf_{x'\in G(u,Y',Z')}\left\Vert x-x'\right\Vert_E^2dud\mathbb{P}\\
&=&\int_A\sup_{x\in G(u,y,z)}\inf_{x'\in G(u,y',z')}\left\Vert x-x'\right\Vert_E^2dud\mathbb{P}\\
&=&\int_A\tilde{\varrho}^2(G(u,y,z),G(u,y',z'))dud\mathbb{P},
\end{eqnarray*}
where $\tilde{\varrho}\left( A,\,B\right):=\displaystyle{\sup_{ a\in A}}\;\displaystyle{\inf_{ b\in B}}\left\Vert a-b\right\Vert_E$ for any $A,B\in\mathscr{K}_{cmpt}(E)$.
\\On the other hand, 
\begin{eqnarray*}
\left\Vert Y-Y'\right\Vert_{L^2\left(\Omega;\gamma(0,T;E)\right)}=\left\Vert Y-Y'\right\Vert_{L^2\left(\Omega;L^2([0,T];E)\right)}=\left\Vert y-y'\right\Vert_E\left((\lambda\otimes\mathbb{P})(A)\right)^{\frac{1}{2}}.
\end{eqnarray*}
Then,
$$\int_A\tilde{\varrho}^2(G(u,y,z),G(u,y',z'))dud\mathbb{P}\leq 2 K_1^2\int_A\left(\left\Vert y-y'\right\Vert_E^2+\left\Vert z-z'\right\Vert_E^2\right)dud\mathbb{P}.$$
It follows that
$$\tilde{\varrho}(G(u,y,z),G(u,y',z'))\leq 2^{\frac{1}{2}} K_1\left(\left\Vert y-y'\right\Vert_E+\left\Vert z-z'\right\Vert_E\right)\quad a.e.(u,\omega).$$
Therefore $(\Gamma_2)-(ii)$ is obtained according to the separability of the space $E$.\bigskip\\
Conversely, assume that $(\Gamma_2)$ is fulfilled.\\
Since $G(.,0_E,0_E)$ is $\Sigma_{\mathbb{F}}$-measurable, then $(\Gamma_2)-(i)$ allows us to ensure that $S^2_{\Sigma_{\mathbb{F}}}(G(.,0_E,0_E))$ is non-empty. Therefore $(\Gamma_1)-(ii)$ is obtained by virtue of Lemma \ref{lem303} and (\ref{ident406}).\\
Let $Y,Z,Y',Z'$ be E-valued $\Sigma_{\mathbb{F}}$-measurable processes defining elements of $L_{\mathbb{F}}^2\left(\Omega;\gamma(0,T;E)\right)$.\\
Thanks to Lemma \ref{lem401}, the set-valued  $G(.,Y,Z)$ is $\Sigma_{\mathbb{F}}$-measurable. Moreover, 
\begin{eqnarray*}
\varrho(G(u,Y(u,\omega),Z(u,\omega)),\left\lbrace 0_E\right\rbrace)&\leq & \varrho(G(u,0_E,0_E),\left\lbrace 0_E\right\rbrace)+\varrho\left(G(u,Y(u,\omega),Z(u,\omega)),G(u,0_E,0_E)\right)\\
&\leq & \left\vert G(u,0_E,0_E)\right\vert_{\varrho}+K_2\left(\left\Vert Y(u,\omega)\right\Vert_E+\left\Vert Z(u,\omega)\right\Vert_E\right),
\end{eqnarray*}
for a.e. $(u,\omega)$. Thus $S^2_{\Sigma_{\mathbb{F}}}(G(.,Y,Z))$ is non-empty, and by Lemma \ref{lem303} the set $\Lambda^{Y,Z}$ so is.\\
Let $g$ be a $\Sigma_{\mathbb{F}}$-measurable process defining element of $\Lambda^{Y,Z}$.\\
Keeping in mind Lemma \ref{lem402}, we infer that there exists a $\Sigma_{\mathbb{F}}$-measurable process $h$ selection of the set-valued $\mathbf{G}(.,Y',Z')$ such that for any $(u,\omega)\in [0,T]\times \Omega$ we have:
\begin{eqnarray*}
d_{E}\left(g(u,\omega),G(u,Y'(u,\omega),Z'(u,\omega))\right)=\left\Vert g(u,\omega)-h(u,\omega)\right\Vert_{E}.
\end{eqnarray*}
It follows
\begin{eqnarray*}
\left\Vert g(u,\omega)-h(u,\omega)\right\Vert_{E} &\leq &\varrho(G(u,Y(u,\omega),Z(u,\omega)),G(u,Y'(u,\omega),Z'(u,\omega)))\\
&\leq & K_2\left(\left\Vert Y(u,\omega)-Y'(u,\omega)\right\Vert_{E}+\left\Vert Z(u,\omega)-Z'(u,\omega)\right\Vert_{E}\right)
\end{eqnarray*}
for a.e. $(u,\omega)$. Thus,
\begin{eqnarray*}
\left\Vert g-h\right\Vert_{L_{\mathbb{F}}^2(\Omega,\gamma(0,T;E))}\leq  K_2\left(\left\Vert Y-Y'\right\Vert_{L_{\mathbb{F}}^2(\Omega,\gamma(0,T;E))}+\left\Vert Z-Z'\right\Vert_{L_{\mathbb{F}}^2(\Omega,\gamma(0,T;E))}\right).
\end{eqnarray*}
Since $h$ defines an element of $\Lambda^{Y',Z'}$, then we get to
\begin{eqnarray*}
d_{L_{\mathbb{F}}^2(\Omega,\gamma(0,T;E))}(g,\Lambda^{Y',Z'})\leq K_2\left(\left\Vert Y-Y'\right\Vert_{L_{\mathbb{F}}^2(\Omega,\gamma(0,T;E))}+\left\Vert Z-Z'\right\Vert_{L_{\mathbb{F}}^2(\Omega,\gamma(0,T;E))}\right).
\end{eqnarray*}
Consequently, $(\Gamma_1)-(iii)$ is obtained.
\end{proof}

\end{document}